\documentclass[a4paper, 11pt]{article}
\usepackage[arxiv]{optional}

\usepackage[english]{babel}
\usepackage[T1]{fontenc}
\usepackage{amsmath}
\opt{arxiv}{%
\usepackage{amsthm}
}%
\usepackage{amsfonts}
\usepackage{amssymb}
\usepackage[final]{graphicx}
\usepackage{paralist}
\usepackage{hyperref}
\usepackage{tikz}

\usepackage{cleveref}
\usepackage{subfigure}
\usepackage{pifont}
\usepackage{caption}
\usepackage{mathtools}

\opt{draft}{%
\usepackage{refcheck}

}

\opt{arxiv}{%
\theoremstyle{plain}
\newtheorem{thm}{Theorem}[section]
\newtheorem{prop}[thm]{Proposition}
\newtheorem*{prop*}{Proposition}
\newtheorem{lem}[thm]{Lemma}
\newtheorem*{lem*}{Lemma}
\newtheorem{cor}[thm]{Corollary}
\newtheorem*{cor*}{Corollary}

\newtheorem*{example*}{Example}
\newtheorem{conject}[thm]{Conjecture}
\newtheorem*{conject*}{Conjecture}

\theoremstyle{definition}
\newtheorem{defn}[thm]{Definition}
\newtheorem*{defn*}{Definition}
\newtheorem{rem}[thm]{Remark}
\newtheorem*{rem*}{Remark}
}%
\opt{springer}{%

}%

\newcommand{\eps}{\varepsilon}
\renewcommand{\emptyset}{\varnothing}

\DeclareMathOperator{\im}{Im}
\DeclareMathOperator{\ind}{ind}

\DeclareMathOperator{\re}{Re}

\opt{springer}{%
\newcommand{\eop}{\qed}
}
\opt{arxiv}{%
\newcommand{\eop}{}
}

\newcommand{\R}{\ensuremath{\mathbb{R}}}
\newcommand{\C}{\ensuremath{\mathbb{C}}}

\newcommand{\defby}{\mathrel{\mathop:}=}

\newcommand{\conj}[1]{\overline{#1}}

\newcommand{\CC}{\ensuremath{\mathcal{C}}}
\newcommand{\Aa}{\ensuremath{\mathcal{A}}}

\DeclareMathOperator{\intx}{int}

\usepackage{color}

\begin{document}

\title{How constant shifts affect the zeros of certain rational harmonic functions}

\date{May 16, 2018}

\author{J\"org Liesen\footnotemark[1] \and Jan Zur\footnotemark[1]}

\maketitle

\renewcommand{\thefootnote}{\fnsymbol{footnote}}

\footnotetext[1]{TU Berlin, Institute of Mathematics, MA 3-3, Stra{\ss}e des 17. Juni 136, 10623 Berlin, Germany.
\texttt{\{liesen,zur\}@math.tu-berlin.de}}

\renewcommand{\thefootnote}{\arabic{footnote}}

\begin{abstract}
We study the effect of constant shifts on the zeros of rational harmomic functions
$f(z) = r(z) - \conj{z}$. In particular,
we characterize how shifting through the caustics of $f$ changes the number of
zeros and their respective orientations. This also yields insight into the nature
of the singular zeros of $f$. Our results have applications in gravitational lensing theory, 
where certain such functions $f$ represent gravitational point-mass lenses, and a constant
shift can be interpreted as the position of the light source of the lens.
\end{abstract}

\section{Introduction}\label{sect:introduction}

The number and location of the zeros of rational harmonic functions of the form
\begin{align}\label{eq:f}
f(z) = r(z) - \conj{z},
\end{align}
where $r$ is a rational function, have been intensively studied in recent years. 
An important result of Khavinson and Neumann~\cite{KhavinsonNeumann2006} says that 
if $\deg(r)\geq 2$, then $f$ may have at most $5\deg(r)-5$ zeros. As shown by a 
construction of Rhie~\cite{Rhie2003}, this bound on the maximal number of zeros 
is sharp in the sense that for every $n\geq 2$ there exists a rational harmonic 
function as in \eqref{eq:f} with $n=\deg(r)$ and exactly $5n-5$ zeros. 
Several authors have derived more refined bounds on the maximal
number of zeros which depend on the degrees of the numerator and denominator
polynomials of~$r$; see, e.g.,~\cite{LiesenZur2018} and the references given there.

Rhie made her construction in the context of astrophysics, where certain rational
harmonic functions model gravitational lenses based on $n$ point-masses; 
see the Introduction of~\cite{SeteLuceLiesen2015a} for a brief summary of 
Rhie's construction, and~\cite{LuceSeteLiesen2014a} for a detailed analysis. 
Descriptions of the connection between complex analysis and gravitational lensing 
are given, for example, in the 
articles~\cite{DanekHeyrovsky2015,KhavinsonNeumann2008,Petters2010,PettersWerner2010},
and a comprehensive treatment can be found 
in the monographs~\cite{PettersLevineWambsganss2001,SchneiderEhlersFalco1999}.
The function modeling the gravitational point-mass lens is
a special case of \eqref{eq:f}, namely
\begin{equation}\label{eq:r}
f(z) = \conj{z}-r(z),\quad\mbox{where}\quad
r(z)=\sum_{k=1}^n \frac{m_k}{z-z_k}.
\end{equation}
The poles $z_1,\dots,z_n\in\C$ represent the position of the respective point-masses 
$m_1,\dots,m_n>0$ in the lens plane. For a fixed $\eta\in\C$, a solution of $f(z)=\eta$,
or equivalently a zero of $f_\eta(z)=f(z)-\eta$, represents a lensed image
of a light source at the position $\eta$ in the source plane. Of great importance 
in this application is the behavior of the zeros under movements of the light source, 
i.e., changes of the parameter $\eta$. Using explicit computations, Schneider and Weiss 
studied this behavior for two point-masses, i.e., $n=2$ in \eqref{eq:r}, in their frequently
cited paper~\cite{SchneiderWeiss1986}. The same model was analyzed extensively 
by Witt and Petters~\cite{WittPetters1993}. Schneider, Ehlers and Falco pointed out 
in~\cite[p.~265]{SchneiderEhlersFalco1999}, that the two point-mass lens is already 
fairly complicated to analyze in detail.  Petters, Levine and Wambsganss gave a more 
general analysis in~\cite[Part~III]{PettersLevineWambsganss2001} based on the Taylor 
series of the gravitational lens potential associated with the lensing map $z\mapsto f_\eta(z)$. 
By truncating the Taylor series and neglecting higher order terms, 
they obtained an approximation to the lensing map's local quantitative 
behavior in~\cite[Section~9.2]{PettersLevineWambsganss2001}.

In this paper we give a rigorous analysis of the effect of varying 
the parameter $\eta$ on the zeros of rational harmonic functions 
of the form
$$f_\eta(z)=f(z)-\eta,\quad \mbox{where $f$ is as in \eqref{eq:f}.}$$
In particular, we study 
the behavior of the zeros when $\eta$ crosses a caustic of $f$ (see 
Section~\ref{sect:mathematical_background} for a definition of this term). 
Apart from advancing the overall understanding of rational harmonic functions,
our goal is to confirm and generalize the above mentioned results published 
in the astrophysics literature. One of the consequences of our findings 
is that may change the number of zeros of $f_\eta$  
by $4\deg(r)-6$. Thus, the effect of varying $\eta$ is considerably
different from the effect of perturbing $f$ by poles that was studied 
in~\cite{SeteLuceLiesen2015a}. 

The paper is organized as follows. In \Cref{sect:mathematical_background} we
discuss the mathematical background, in particular the
critical curves, caustics, and exceptional points (zeros and poles) of $f$.
In \Cref{sect:main_section_1} we focus on constant shifts that do not
affect the number of zeros. Our main results are contained in \Cref{sect:main_section_2},
where we study in detail how shifting $\eta$ across a caustic of $f$ affects
the zeros. Here we distinguish between shifting through fold and cusp points
of $f$. Our results on shifts also yield some insight 
into the nature of the singular zeros of $f$. In \Cref{sect:examples} we give 
examples that illustrate our results and a brief outlook on possible 
extensions and further work in this area.

\section{Critical curves, caustics and the Poincar\'e index}\label{sect:mathematical_background}

Let a rational harmonic function $f(z) = r(z) - \conj{z}$ with $\deg(r)\geq 2$ be given. Using
the Wirtin\-ger derivatives $\partial_z$ and $\partial_{\conj{z}}$ we can write the \emph{Jacobian}
of $f$ as
$$J_f(z) = |\partial_z f(z)|^2-|\partial_{\conj{z}} f(z)|^2=|r'(z)|^2-1.$$
The points $z\in \C$ where $J_f$ vanishes, i.e., where $|r'(z)| = 1$, are called the \emph{critical points} of $f$.
We denote the set of the critical points by $\CC$. The critical points of $f$ are the preimages of the unit circle
$|w|=1$ under the map $w=r'(z)$, which is analytic (and non-constant) in $\C$, except at the finitely many poles
of $r'(z)$. Thus, the critical points form finitely many closed curves that separate the complex plane into regions
where $J_f(z)>0$ and hence $f$ is \emph{sense-preserving}, and $J_f(z)<0$ and hence $f$ is \emph{sense-reversing}.
We denote these regions by $\Omega_+$ and $\Omega_{-}$, respectively, so that we have the disjoint partitioning $\C = \CC \cup \Omega_+ \cup \Omega_-$.
Each closed curve in the set $\CC$ is called a \emph{critical curve} of $f$.

The necessary condition for a stationary point of the Jacobian of $f$ is
$$\partial_{\conj{z}} J_f(z)=r'(z)\conj{r''(z)}=0,$$
and hence the condition $r''(z) \neq 0$ for all $z\in\CC$ implies that no critical point of $f$ is a saddle-point
of $J_f$. Then the critical curves of $f$ are smooth Jordan curves, and in particular they do not intersect
each other; see the left plot in Figure~\ref{fig:dummyfig} for an example. A function $f$ with this property
is called \emph{non-degenerate}, and in the following we will always assume
that the given $f$ is such a function.

In this case the critical curves yield a disjoint partitioning of
$\Omega_+\cup \Omega_-$ into finitely many
open and connected subsets $A_1,\dots,A_m$, where $\partial A_j \subseteq \CC$, and
either $A_j\subseteq \Omega_+$ or $A_j\subseteq \Omega_{-}$, for $j=1,\dots,m$, and we write
\begin{equation}\label{eq:setA}
\Aa \defby \left\{A_1,\dots, A_m\right\}.
\end{equation}
Exactly one of the sets $A_j$ is unbounded, and we sometimes denote this set by $A_\infty$.
On the two bordered sets of a given critical curve is $f$ always differently oriented. This is a consequence
of the maximum modulus principle applied to the functions $r'$ and $1/r'$.

\begin{figure}
\begin{center}
\includegraphics[width=1.0\textwidth]{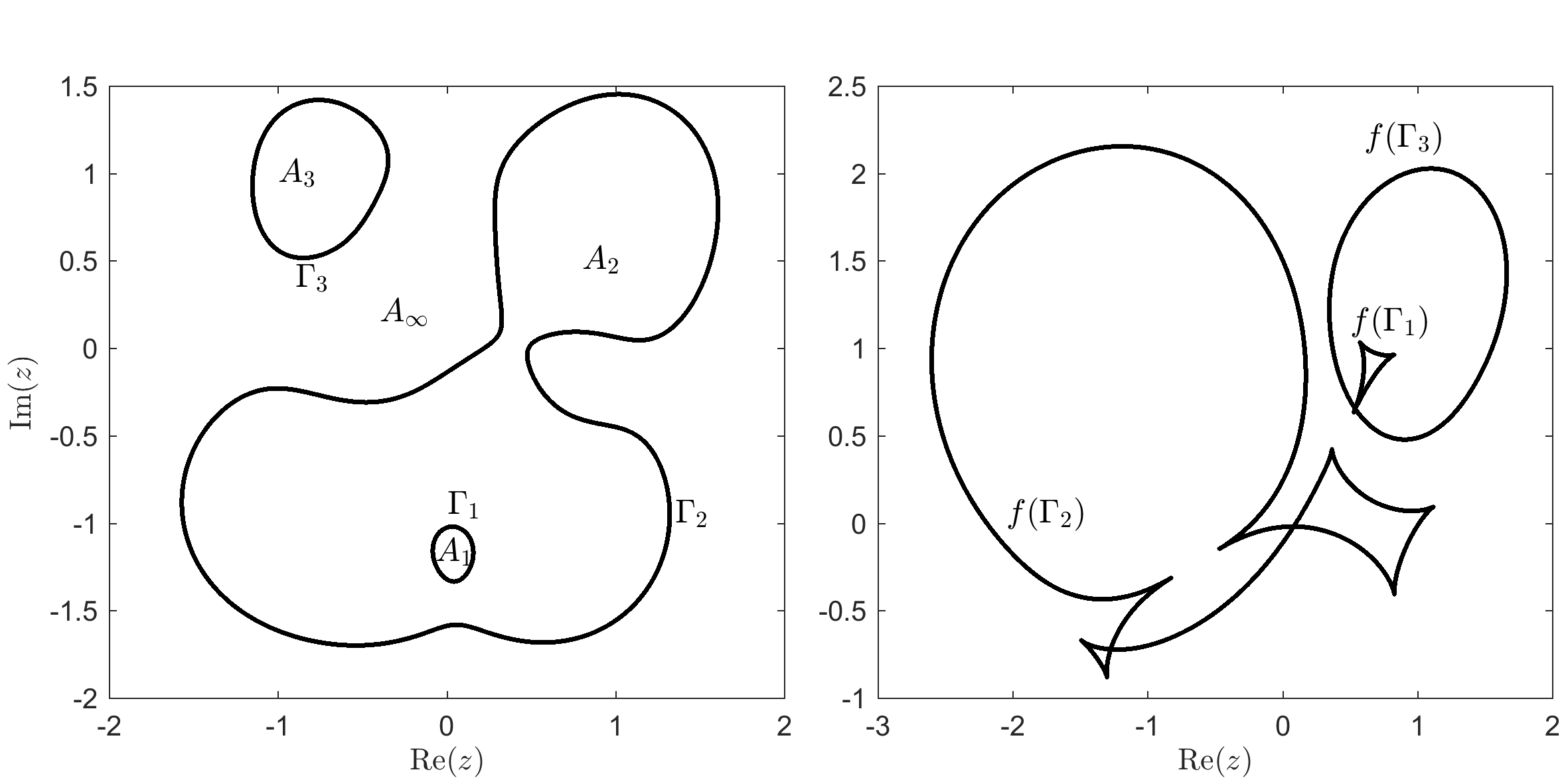}
\caption{Critical curves $\Gamma_i$ and $\mathcal{A}$ (left), caustics $f(\Gamma_i)$ (right)}\label{fig:dummyfig}
\end{center}
\end{figure}

The elements of the set $f(\CC)$ are called the \emph{caustic points} of $f$, and for each critical curve $\Gamma$,
the curve $f(\Gamma)$ is called a \emph{caustic} of $f$. Unlike a critical curve, a caustic of $f$ may intersect itself
as well as other caustics of $f$, and a caustic of $f$ need not be smooth; see the right plot of
Figure~\ref{fig:dummyfig} for examples.

The singularities on a caustic of $f$ are called \emph{cusp points}, and all other caustic points of $f$
are called \emph{fold points}; cf. \cite[p.~88]{PettersLevineWambsganss2001}. In order to characterize a
cusp point, note that the unique tangent at a critical point $z_0 \in \CC$ is given by
\begin{align}\label{eq:tangent}
g(t) = h t+z_0,\quad h \defby \frac{ir'(z_0)\conj{r''(z_0)}}{|r'(z_0)\conj{r''(z_0)}|},
\end{align}
where we use that the gradient of the Jacobian is orthogonal with respect to its contour 
line, i.e., the critical curve, and where the normalization of the direction~$h$ will
be convenient in our derivations in Section~\ref{sect:main_section_2}. The linearization of $f$ at $z_0\in\CC$ has the form
$$L_{z_0}(z) \defby r(z_0) + e^{i\varphi}(z - z_0) - \conj{z},$$ 
where we use that $r'(z_0) = e^{i\varphi}$ for some $\varphi \in [0,2\pi)$. After some small manipulations we obtain
\begin{align*}
L_{z_0}(g(t))=2 i e^{i\varphi/2}\im(e^{i\varphi/2} h)t+(r(z_0)+e^{i\varphi}z_0-\conj{z}_0),
\end{align*}
which shows that the tangent direction at the caustic point $f(z_0)\in f(\CC)$ is given
by $ie^{i\varphi/2}$. Moreover, $L_{z_0}$ maps the tangent at the critical curve to a 
single point if and only if 
$$\im(e^{i\varphi/2}h) = \im(ie^{3i\varphi/2}\conj{r''(z_0)})=0,$$ 
or, equivalently,
\begin{align}\label{eq:cusp_condition}
\re\left(\frac{r''(z_0)}{r'(z_0)^{3/2}}\right) = 0 \quad\Longleftrightarrow\quad
\re\left(\frac{r'(z_0)^{3/2}}{r''(z_0)}\right) = 0,
\end{align}
where the equivalence is defined since we assume that $f$ is non-degenerate.
Let us summarize these considerations.

\begin{lem}\label{lem:cusp}
Let $z_0\in\CC$ be a critical point of $f(z)=r(z)-\conj{z}$. Then the caustic point $f(z_0) \in f(\CC)$ is a cusp point if and only if \eqref{eq:cusp_condition} holds. 
(Each other caustic point $f(z_0)\in f(\CC)$ is called a fold point.)
\end{lem}

Petters and Witt~\cite{PettersWitt1996} showed that if $r$ is as in \eqref{eq:r}, then 
there can be at most $12n(n-1)$ cusp points; see 
also~\cite[Section~15.3.3]{PettersLevineWambsganss2001}. The determination of a sharp 
upper bound on the number of cusp points was mentioned as an open research problem 
in~\cite{Petters2010}. The relation between the number of cusp points
and the number of zeros for harmonic \emph{polynomials} was recently studied in \cite{KhavinsonLeeSaez2015}.

Now let $z_0\in\C$ be such that $f(z_0)=r(z_0)-\conj{z}_0=0$. We call $z_0$ a \emph{sense-preserving},
\emph{sense-reversing}, or \emph{singular zero} of $f$, if $z_0$ is an element of $\Omega_+$, $\Omega_{-}$,
or $\CC$, respectively. Note that if $z_0$ is a sense-preserving or sense-reversing zero,
then there exists an $\eps>0$ such that $f$ is sense-preserving or sense-reversing, respectively,
on $B_{\eps}(z_0)$, the open disk around $z_0$ with radius $\eps$.
The sense-preserving and sense-reversing zeros of $f$ are also called the \emph{regular zeros}
of $f$. If $f$ has only such zeros, $f$ is called \emph{regular}, and otherwise $f$ 
is called \emph{singular}.

We have the following simple but important relation between singular zeros and caustic points.

\begin{prop}\label{prop:critcaus}
Let $f(z) = r(z) - \conj{z}$ with $\deg(r)\geq 2$ and $\eta \in \C$ be given.
Then $f_\eta(z) \defby f(z) - \eta$ has a singular zero if and only if $\eta$
is a caustic point of $f$.
\end{prop}

\begin{proof}
If $z_0\in\C$ is a singular zero of $f_\eta$, then $z_0\in\CC$ and $f(z_0)-\eta=0$, or
$f(z_0)=\eta$, which means that $\eta$ is a caustic point of $f$. On the other hand, if
$\eta\in\C$ is a caustic point of $f$, then $f(z_0)=\eta$ for some $z_0\in\CC$, which means
that $z_0$ is a singular zero of $f_\eta$.
\eop
\end{proof}

Sometimes we will use the contraposition of the statement of Proposition~\ref{prop:critcaus}:
If $\eta\in \C$ is \emph{not} on a caustic of $f$, i.e., $\eta\notin f(\CC)$, then
the shifted function $f_\eta(z)=f(z)-\eta$ does \emph{not} have a singular zero
and hence is regular.

Let us briefly recall the argument principle for continuous functions;
see~\cite[Corollary~2.6]{Balk1991},~\cite[Theorem 2.2]{SuffridgeThompson2000},
or~\cite[Section~2]{SeteLuceLiesen2015a} for more details.
Let $\Gamma$ be a closed Jordan curve, and let $f$ be a function
that is continuous and nonzero on $\Gamma$. Then the \emph{winding} of $f$ on~$\Gamma$ is defined as the
change in the argument of $f(z)$ as $z$ travels once around $\Gamma$ in the positive direction,
divided by $2\pi$, i.e., 
$$V(f;\Gamma) \defby \frac{1}{2\pi}\Delta_\Gamma \arg f(z).$$
A point $z_0\in\C$ is called an \emph{exceptional point} of a function $f$, if $f$ is either zero,
not continuous, or not defined at $z_0$. If $f$ is continuous and nonzero in a punctured neighborhood $D$ of
an exceptional point $z_0$, and hence the exceptional point $z_0$ is isolated, then the \emph{Poincar\'e index}
of $f$ at $z_0$ is defined as $\ind(f;z_0)\defby V(f;\Gamma)$, where $\Gamma$ is an arbitrary
closed Jordan curve in $D$ and around $z_0$. This can be seen as a generalization
of the order of a zero or a pole of a meromorphic function; cf. \cite[Example~2.5]{SeteLuceLiesen2015a}.
The Poincar\'e index is independent of the choice of the Jordan curve $\Gamma$, as long as $z_0$ is
the only exceptional point of $f$ in $\intx(\Gamma)$, the interior of $\Gamma$.
If there are several (isolated) exceptional points
in $\intx(\Gamma)$, we have the following theorem.

\begin{thm}\label{th:winding}
If $\Gamma$ is a closed Jordan curve and the function $f$ is continuous and
nonzero on $\overline{\intx(\Gamma)}$ except for finitely many exceptional points
$z_1,\dots,z_k \in \intx(\Gamma)$, then
\begin{align*}
V(f;\Gamma) = \sum\limits_{j = 1}^k \ind(f;z_j).
\end{align*}
\end{thm}

For the functions of our interest, which are continuous in $\C$ except for finitely many exceptional
points, we have the following Poincar\'e indices; see~\cite[Proposition~2.7]{SeteLuceLiesen2015a}.

\begin{prop}\label{prop:poincare_index}
Let $f(z) = r(z) - \conj{z}$ with $\deg(r)\geq 2$ be given. The Poincar\'e index of $f$ at a sense-preserving
zero is $+1$, and at a sense-reversing it is $-1$. If $z_0$ is a pole of $r$ of order $m$, then $f$ is
sense-preserving in a neighborhood of $z_0$, and the Poincar\'e index of $f$ at $z_0$ is $-m$.
\end{prop}

The determination of the Poincar\'e index of a singular zero is more challenging. For the functions of our
interest it may be $-1$, $0$, or $+1$ (see Corollary~\ref{cor:singular} and its discussion), while for 
a general harmonic function
it may even be undefined; see~\cite[p.~413]{DurenHengartnerLaugesen1996}.

The next result, which is an immediate consequence of Theorem~\ref{th:winding} and Proposition~\ref{prop:poincare_index},
shows how we can use the argument principle in order to determine the number of zeros.

\begin{cor}\label{cor:zero_counting}
Let $f(z) = r(z) - \conj{z}$ with $\deg(r) \ge 2$ be given. If $f$ is nonzero on a closed
Jordan curve $\Gamma$ and has no singular zero in $\intx(\Gamma)$, then
\begin{align*}
V(f;\Gamma) = N_+(f;\intx(\Gamma)) - N_-(f;\intx(\Gamma)) - P(f;\intx(\Gamma)),
\end{align*}
where $N_{(+,-)}(f;\intx(\Gamma))$ denotes the number of sense-preserving and sense-reversing zeros,
and $P(f;\intx(\Gamma))$ denotes the number of poles (with multiplicities) of $f$ in $\intx(\Gamma)$.
\end{cor}

Finally, we state a version of Rouch{\'e}'s theorem which we will frequently use
in order to decide whether two functions have the same winding on a given Jordan curve.
A short proof of this result is given~\cite[Theorem 2.3]{SeteLuceLiesen2015a}.

\begin{thm}\label{thm:Rouche}
Let $\Gamma$ be a closed Jordan curve and suppose that $f,g : \Gamma \rightarrow \C$ are continuous.
If $|f(z) - g(z)| < |f(z)|+|g(z)|$ holds for all $z \in \Gamma$, then $V(f;\Gamma) = V(g;\Gamma)$.
\end{thm}

\section{Constant shifts that do not affect the number of zeros}\label{sect:main_section_1}

In this section we will begin our study of the effect of constant shifts on the zeros
of a given \emph{non-degenerate} rational harmonic function
\begin{equation}\label{eqn:def_f}
f(z)=r(z)-\conj{z}\quad\mbox{with}\quad \deg(r)\geq 2\quad\mbox{and}\quad
\lim_{|z|\rightarrow\infty}|f(z)|=\infty.
\end{equation}
As mentioned in~\cite[Remark~3.2]{SeteLuceLiesen2015a}, the assumption 
$\lim_{|z|\rightarrow\infty}|f(z)|=\infty$ is not restrictive. It only excludes functions
with $r(z)=\alpha z + p(z)/q(z)$, where $|\alpha|=1$ and $\deg(p)\leq \deg(q)$.

In addition to the notation established in Corollary~\ref{cor:zero_counting}, we denote by $N(f;A)$
the number of zeros of $f$ in the set $A\subseteq \C$, and write $N(f) \defby N(f;\C)$ for
brevity. Moreover, by $N_s(f)$ we denote the number of singular zeros of $f$.

Our first result characterizes the zeros of the shifted function $f_\eta(z) = f(z)-\eta$
for a \emph{sufficiently large} (real) shift $\eta>0$.

\begin{thm}\label{thm:initial_value}
Let $f$ be as in \eqref{eqn:def_f} with $r(z)= \frac{p(z)}{q(z)}$,
let $v_1, \dots, v_{m}$ be the poles of $f$ with their respective
multiplicities $\mu_1,\dots,\mu_m$, and let
$k \defby \max(\deg(p) - \deg(q),1)$. If $\eta>0$ is sufficiently large,
then $f_\eta$ has exactly $\deg(q) + k$ zeros $z_1, \dots z_{\deg(q) + k}$, and it exists an $\eps > 0$ such that
\begin{compactenum}[(i)]
\item{$N(f_{\eta};B_\eps(v_j))
=N_{+}(f_{\eta};B_\eps(v_j))= \mu_j$ \quad  for $j=1,\dots,m$,}
\item{$z_j\in A_\infty\setminus
\bigcup_{\ell=1}^m B_\eps(v_\ell)$ \quad for $j = \deg(q) + 1, \dots, \deg(q) + k$.}
\end{compactenum}
\end{thm}

\begin{proof}
In order to explain the general idea of the proof, assume that we are given
some $\eta \gg 1$. Then $f_\eta(z_j)=0$ means that
$r(z_j) - \conj{z}_j\gg 1$, which can happen when the zero $z_j$ of $f_\eta$ is close
to a pole of $f$, or when $|z_j| \gg 1$. These cases correspond to (i)
and (ii), and we will now first prove the existence of the zeros in (i),
and then of the additional zeros in (ii).

\emph{Case 1 (zeros close to a pole):} In the neighborhood of any pole $v_j$
of $f$ we have $|r'(z)|\gg 1$, and hence $f$ is sense-preserving. Therefore we
can find an $\eps > 0$ such that
\begin{compactenum}[(a)]
\item{$f$ is sense-preserving on $B_\eps(v_j)$ for all $j=1,\dots,m$,}
\item{$B_\eps(z_{j}) \cap B_\eps(z_{\ell}) = \emptyset$ for all
$j,\ell \in \{1,\dots,m\}$ with $j \neq \ell$.}
\end{compactenum}
Now consider any $\eta \geq$ $2 \max\{|f(z)| : z \in \bigcup_{j=1}^{m} \partial B_\eps(v_j) \} > 0$,
and let $g(z)\equiv -\eta$. Then
\begin{align*}
|f_{\eta}(z) - g(z)| = |f(z)| < |\eta|\le |f_{\eta}(z)| + |g(z)|  \quad \text{ for all } z \in \bigcup_{j = 1}^m \partial B_\eps(v_j).
\end{align*}
Since $V(g;\partial B_\eps(v_j))=0$, the function $f$ has a pole of order
$\mu_j$ in $B_\eps(v_j)$, and $f$ is sense-preserving in $B_\eps(v_j)$,
Theorem~\ref{thm:Rouche} and Proposition~\ref{prop:poincare_index} yield
$$N(f_{\eta};B_\eps(v_j))=N_{+}(f_{\eta};B_\eps(v_j)) = \mu_j
\quad \mbox{for all $j=1,\dots,m$,}$$
which proves the existence of the zeros $z_1,\dots, z_{\deg(q)}$ as stated in (i).

\emph{Case 2 (zeros away from the poles):} We need to distinguish four cases according to the
degrees of $p$ and $q$.

\emph{(a)} $\deg(p) < \deg(q)$, hence $k=1$: In this case
$\lim_{|z|\rightarrow \infty} |r(z)| = 0$. Therefore, if $\eta>0$
is chosen large enough, there exists a $\delta > 0$, such that
$B_\delta(-\eta) \subset A_\infty$, and we have $|r(z)| < \delta$
for all $ z \in \partial B_\delta(-\eta)$, as well as
$B_\delta(-\eta)\cap B_\eps(v_j)=\emptyset$ for
all $j=1,\dots,m$. Using the function $g(z) \defby -\conj{z} - \eta$,
which has $-\eta$ as its only zero, we obtain
\begin{align*}
|f_{\eta}(z) - g(z)| = |r(z)| < \delta = |g(z)| \le |f_{\eta}(z)| + |g(z)|
\text{ for all } z \in \partial B_\delta(-\eta).
\end{align*}
Using Theorem~\ref{thm:Rouche} and Proposition~\ref{prop:poincare_index} we get
$$V(f_{\eta};\partial B_\delta(-\eta)) = V(g;\partial B_\delta(-\eta))=-1,$$
since in our region of interest $|r'(z)|<1$, and therefore
$N(f_{\eta};B_\delta(-\eta)) = N(g; B_\delta(-\eta)) = 1$. This shows the
existence of one additional zero of $f_\eta$, which is contained in the 
set $A_\infty$.

\emph{(b)} $\deg(p) = \deg(q)$, hence $k=1$: In this case we have $p=cq+\widetilde{q}$
for some (nonzero) $c\in\C$ and some polynomial $\widetilde{q}$ with $\deg(\widetilde{q})<\deg(q)$.
Hence $r(z)=c+\widetilde{q}(z)/q(z)$, where
$\lim_{z\rightarrow\infty}|\widetilde{q}(z)/q(z)|=0$.
We can now apply the same argument as in the previous case with
the function $g(z) \defby -\conj{z}-\eta+c$, using the disk $B_\delta(-\eta+\conj{c})$ for sufficiently small $\delta > 0$.

In the next two cases we will use that whenever $\deg(p)-\deg(q)>0$,
we can write
\begin{align}\label{eq:partial_fraction_decomposition}
r(z) = cz^k + \rho(z) + \widetilde{r}(z), \quad k=\deg(p)-\deg(q),
\end{align}
for some (nonzero) $c\in\C$, some polynomial $\rho$ of degree at most $k-1$, and
some rational function $\widetilde{r}$ with
$\lim_{|z|\rightarrow \infty}|\widetilde{r}(z)| = 0$.

\emph{(c)} $\deg(p) = \deg(q) + 1$, hence $k=1$: 
Our general assumption $\lim_{|z|\rightarrow\infty}|f(z)|=\infty$ implies that
in this case we have \eqref{eq:partial_fraction_decomposition} with
$|c|\neq 1$. We will first show that for each $\eta$
the function $g(z) \defby cz - \eta - \overline{z}$ has exactly one zero. 
Writing $c = \re(c) + i\im(c) $ and $z = x + iy$, the equation $g(z)=0$
can be written as
\begin{align*}
\begin{bmatrix}
\re(c)-1 & -\im(c) \\
\im(c) &  \re(c)+1
\end{bmatrix}
\begin{bmatrix}
x \\
y
\end{bmatrix}
=
\begin{bmatrix}
\eta\\
0
\end{bmatrix}.
\end{align*}
The determinant of the matrix is $\re(c)^2 -1 + \im(c)^2  = |c|^2-1\neq 0$.
Denoting the unique zero of $g$ by $z_g$ and using that
$\lim_{|z|\rightarrow \infty}|\widetilde{r}(z)| = 0$, we can choose $\eta>0$
sufficiently large so that $|\widetilde{r}(z)| < |cz-\eta -\overline{z}|$
holds for some $\delta>0$ and all $ z \in \partial B_\delta(z_g)$. As above
we can assume that $B_\delta(z_g)\subset A_\infty$ and that
$B_\delta(z_g)\cap B_\eps(v_j)=\emptyset$ for
all $j=1,\dots,m$. We then get
\begin{align*}
|f_{\eta}(z) - g(z)| = |\widetilde{r}(z)| <
|cz-\eta - \conj{z}| = |g(z)| \le |f_{\eta}(z)| + |g(z)|
\end{align*}
for all $z \in \partial B_\delta(z_g)$, so that
$N(f_{\eta};B_\delta(z_g)) = N(g; B_\delta(z_g)) = 1$ follows
from Theorem~\ref{thm:Rouche} and Proposition~\ref{prop:poincare_index}.

\emph{(d)} $\deg(p) > \deg(q)$ + 1, hence $k>1$: For a given $\eta>0$,
let $\eta_1,\dots,\eta_k$ be the $k\geq 2$ zeros
of $g(z) \defby cz^k - \eta$. Using \eqref{eq:partial_fraction_decomposition}
we can choose $\eta>0$ sufficienty large and $\delta > 0$ so that
$B_\delta(\eta_j)\subset A_\infty$ for all all $j=1,\dots,k$, and
$B_\delta(\eta_j)\cap B_\eps(v_\ell)=\emptyset$ for all $j\neq \ell$,
as well as $|\rho(z) + \widetilde{r}(z) - \conj{z}| < |g(z)|$
for all $z \in \partial B_\delta(\eta_j)$, $j = 1, \dots, k$. Therefore
\begin{align*}
|f_{\eta}(z) - g(z)| = |\rho(z) + \widetilde{r}(z) - \conj{z}| <  |g(z)| \le |f_{\eta}(z)| + |g(z)|
\end{align*}
for all $z \in \partial B_\delta(\eta_j)$, and the application of
Theorem~\ref{thm:Rouche} and Proposition~\ref{prop:poincare_index} finishes the proof.
\eop
\end{proof}

At the end of this section we will show that the assertions of Theorem~\ref{thm:initial_value}
hold for every $\eta\in\C$ with $|\eta|$ large enough.

We next prove that a \emph{sufficiently small} shift $\eta\in\C$ changes neither
the number nor the orientation of the regular zeros of $f$. This result is a
slight extension of~\cite[Lemma~2.5]{LuceSeteLiesen2014b}.

\begin{thm}\label{thm:small_perturbation}
Let $f$ be as in \eqref{eqn:def_f}, and let $z_1,\dots,z_m$ be the regular,
and $z_{m+1},\dots,z_M$ be the singular zeros of $f$. Let further
$\eps > 0$ be such that 
$\left(\bigcup_{j = 1}^m B_{\eps}(z_j)\right)\cap\CC=\emptyset$, 
and $B_{\eps}(z_j) \cap B_{\eps}(z_k) = \emptyset$
for all $j,k \in \{1,\dots,M\}$ with $j \neq k$. If $\eta\in\C$ satisfies
\begin{align*}
|\eta|<\delta \defby \min \Big\{|f(z)| : z \in \C \setminus \bigcup_{j = 1}^M B_{\eps}(z_j)\Big\},
\end{align*}
then the following properties hold:
\begin{compactenum}[(i)]
\item For each $j=1,\dots,m$ the functions $f$ and $f_\eta$ have the same orientation on
$B_{\eps}(z_j)$, and $N(f;B_{\eps}(z_j))=N(f_\eta;B_{\eps}(z_j))=1$.
\item $N(f;\C \setminus \bigcup_{j = m+1}^M B_{\eps}(z_j)) =
N(f_{\eta}; \C \setminus \bigcup_{j = m+1}^M B_{\eps}(z_j))=m$.
\end{compactenum}
\end{thm}

\begin{proof}
\emph{(i)} From the construction it is clear that $f$ and $f_\eta$ have the
same orientation on each set $B_{\eps}(z_j)$. Moreover,
for all $z \in \partial B_{\eps}(z_j)$ we have
\begin{align}\label{eq:fmfe}
|f(z) - f_\eta(z)| =  |\eta| < \delta \le |f(z)| \le |f(z)| + |f_\eta(z)|,
\end{align}
and hence $V(f; \partial B_{\eps}(z_j) ) = V(f_\eta; \partial B_{\eps}(z_j) )$ by Theorem~\ref{thm:Rouche}.
Since $f$ has exactly one zero in $B_{\eps}(z_j)$, and the poles of $f$ and $f_\eta$ coincide, the assertion
follows from Corollary~\ref{cor:zero_counting}.

\emph{(ii)} We know from (i) that $f_\eta$ has exactly one zero in in each of the
sets $B_{\eps}(z_j)$, $j=1,\dots,m$. If $f_\eta$ has an additional zero
$z_0 \in \C\setminus \bigcup_{j = m+1}^M B_{\eps}(z_j)$, then we can choose an $\eps_0>0$
such that $f$ is nonzero on $\overline{B_{\eps_0}(z_0)}$. Then \eqref{eq:fmfe} holds for all
$z\in \partial B_{\eps_0}(z_0)$, and Theorem~\ref{thm:Rouche} yields
$V(f; \partial B_{\eps_0}(z_0) ) = V(f_\eta; \partial B_{\eps_0}(z_0) )$, which is a contradiction,
since $f$ and $f_\eta$ have the same number of poles, but $f$ has no zeros in $B_{\eps_0}(z_0)$.
\eop
\end{proof}

Note that our general assumption $\lim_{|z|\rightarrow\infty}|f(z)|=\infty$ implies
that $\delta>0$ in Theorem~\ref{thm:small_perturbation}.

Our next goal is to show that the number of zeros of the shifted functions $f_\eta$ remains constant
as long as the shift $\eta$ does not cross a caustic of $f$. Our proof is based on the following two lemmas.

\begin{lem}\label{lem:main}
If $\Gamma$ is a critical curve of $f$, and $\eta_1,\eta_2 \in \C$ are such that
$\lambda\eta_1 + (1-\lambda)\eta_2 \not \in f(\Gamma)$ holds for all $0 \le \lambda \le 1$,
then $V(f_{\eta_1};\Gamma) = V(f_{\eta_2};\Gamma)$.
\end{lem}

\begin{proof}
Using an appropriate rotation and translation of the complex plane we may assume without loss of generality
that $\eta_1 > 0$ and $\eta_2 = 0$. Our assumption then reads $\lambda\eta_1 \not \in f(\Gamma)$ for
all $0\leq \lambda\leq 1$, and Proposition~\ref{prop:critcaus} implies that
$f_{\lambda\eta_1}(z)=f(z)-\lambda\eta_1\neq 0$ holds
for all $0\leq \lambda\leq 1$ and $z\in\Gamma$.

By construction and the triangle inequality we have
\begin{align}\label{eq:tmp1}
\eta_1 = |f(z) - f_{\eta_1}(z)| \le |f(z)| + |f_{\eta_1}(z)| \quad \text{ for all } z \in \C.
\end{align}
If equality holds in \eqref{eq:tmp1} for some $z_0 \in \Gamma$, then $|f(z_0)|<\eta_1$,
since $f_{\eta_1}(z_0)\neq 0$. Moreover,
\begin{align*}
\eta_1 - |f(z_0)| = |f_{\eta_1}(z_0)| = |\eta_1 - f(z_0)|,
\end{align*}
which implies, together with $|f(z_0)| < \mu \eta_1$, that $f(z_0)=\mu \eta_1$ for some $0<\mu<1$. But this means that $f(z_0)-\mu \eta_1=0$
with $z_0\in\Gamma$, i.e., $\mu \eta_1$ is on the caustic $f(\Gamma)$, which is a contradiction. Consequently, we must have a strict
inequality in \eqref{eq:tmp1}, and hence $V(f_{\eta_1};\Gamma) = V(f;\Gamma)$ by Theorem~\ref{thm:Rouche}.
\eop
\end{proof}

\begin{lem}\label{lem:main_2}
If $\eta_1,\eta_2 \in \C$ are such that $\lambda\eta_1 + (1-\lambda)\eta_2 \not \in f(\CC)$
holds for all $0\leq \lambda \leq 1$, then $N(f_{\eta_1};A) = N(f_{\eta_2};A)$ holds for each set
$A\in \Aa$, and $N(f_{\eta_1}) = N(f_{\eta_2})$.
\end{lem}

\begin{proof}
By Proposition~\ref{prop:critcaus}, the functions $f_{\eta_1}$ and $f_{\eta_2}$ are regular. Moreover, these functions have the same
poles, which are equal to of poles of $f$. For a bounded set $A \in \Aa$ we have a unique critical curve $\Gamma^* \subset \partial A$ such that $A\subset\intx(\Gamma^*)$.
If $f$ is sense-preserving on $A$, then Corollary~\ref{cor:zero_counting} and Lemma~\ref{lem:main} imply
\begin{align*}
N(f_{\eta_1};A) &= V(f_{\eta_1};\Gamma^*) - \sum\limits_{\mathclap{\substack{\Gamma \subset \partial A \text{ crit. curve}, \\ \intx(\Gamma)\cap A = \emptyset}}}V(f_{\eta_1};\Gamma) - P(f;\intx(\Gamma)) \\
&= V(f_{\eta_2};\Gamma^*) - \sum\limits_{\mathclap{\substack{\Gamma \subset \partial A \text{ crit. curve}, \\ \intx(\Gamma)\cap A = \emptyset}}}V(f_{\eta_2};\Gamma) - P(f;\intx(\Gamma)) = N(f_{\eta_2};A).
\end{align*}
If $f$ is sense-reversing on $A$, then
\begin{align*}
N(f_{\eta_1};A) &= - V(f_{\eta_1};\Gamma^*) + \sum\limits_{\mathclap{\substack{\Gamma \subset \partial A \text{ crit. curve}, \\ \intx(\Gamma)\cap A = \emptyset}}}V(f_{\eta_1};\Gamma) \\
&= -V(f_{\eta_2};\Gamma^*) + \sum\limits_{\mathclap{\substack{\Gamma \subset \partial A \text{ crit. curve}, \\ \intx(\Gamma)\cap A = \emptyset}}}V(f_{\eta_2};\Gamma) = N(f_{\eta_2};A).
\end{align*}

Using an additional artificial curve $\partial B_R(0)$ for a sufficiently
large $R>0$, containing all zeros and poles of $f_{\eta_1}$ and $f_{\eta_2}$
in its interior,
we obtain the equality for the set $A_\infty$. Finally, $N(f_{\eta_1}) = N(f_{\eta_2})$
holds since $f_{\eta_1}$ and $f_{\eta_2}$ have no singular zeros; cf. Proposition~\ref{prop:critcaus}.
\eop
\end{proof}

Now suppose that $\eta_1, \eta_2 \in \C$ are linked by a continuous path, $\psi: [0,1] \rightarrow \C$, with $\psi(0) = \eta_1$, $\psi(1) = \eta_2$, and $\psi([0,1]) \cap f(\CC) = \emptyset$. Since $\C \setminus f(\CC)$ is open, we can approximate $\psi([0,1])$ arbitrarily closely by a polygonal chain in $\C\setminus f(\CC)$; see \Cref{fig:polygonal_chain} for an illustration. Applying Lemma~\ref{lem:main_2} successively on this chain gives the following result.

\begin{thm}\label{thm:main}
If $\eta_1,\eta_2 \in \C$ are linked by continuous path that does not cross a caustic of $f$,  then $N(f_{\eta_1};A) = N(f_{\eta_2};A)$ holds for each set
$A\in \Aa$, and $N(f_{\eta_1}) = N(f_{\eta_2})$.
\end{thm}

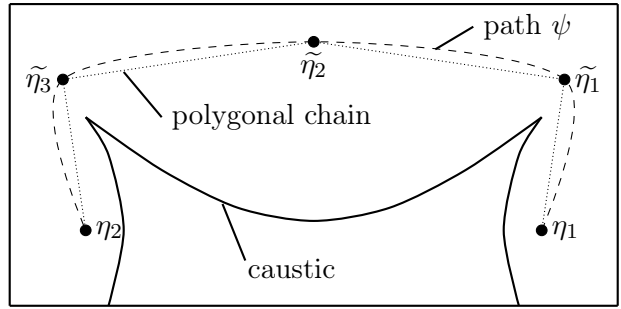
\begin{figure}[h]
\begin{center}
\begin{tikzpicture}
\draw [thick] plot [smooth] coordinates {(0,0) (8,0)};
\draw [thick] plot [smooth] coordinates {(0,4) (0,0)};
\draw [thick] plot [smooth] coordinates {(0,4) (8,4)};
\draw [thick] plot [smooth] coordinates {(8,4) (8,0)};
\draw [thick] plot [smooth] coordinates {(1,2.5) (1.3,2) (1.5,1)(1.3,0)};
\draw [thick] plot [smooth] coordinates {(1,2.5)(2,1.8) (3,1.3) (4,1.125) (5,1.3) (6,1.8) (7,2.5)};
\draw [thick] plot [smooth] coordinates {(7,2.5) (6.7,2)(6.5,1) (6.7,0)};
\draw [dashed] plot [smooth] coordinates {(7,1) (7.3,3) (4,3.5)(.7,3) (1,1)};
\draw [densely dotted] plot coordinates {(7,1) (7.3,3) (4,3.5)(.7,3) (1,1)};
\draw[fill=black](7,1)circle(2pt);
\draw[color=black] (7,1) node[right] {$\eta_1$};
\draw[fill=black](1,1)circle(2pt);
\draw[color=black] (1,1) node[right] {$\eta_2$};
\draw[fill=black](7.3,3)circle(2pt);
\draw[color=black] (7.3,3) node[right] {$\widetilde{\eta_1}$};
\draw[fill=black](4,3.5)circle(2pt);
\draw[color=black] (4,3.5) node[below] {$\widetilde{\eta_2}$};
\draw[fill=black](0.7,3)circle(2pt);
\draw[color=black] (0.7,3) node[left] {$\widetilde{\eta_3}$};
\draw[color=black] (3,.5) node[right] {caustic};
\draw [thick] plot [smooth] coordinates {(2.8,1.35) (3.1,.6)};
\draw[color=black] (6.1,3.7) node[right] {path $\psi$};
\draw [thick] plot [smooth] coordinates {(5.6,3.4) (6.1,3.7)};
\draw[color=black] (2,2.5) node[right] {polygonal chain};
\draw [thick] plot [smooth] coordinates {(2,2.5) (1.5,3.1)};
\end{tikzpicture}
\end{center}
\caption{Illustration of Theorem~\ref{thm:main}.}\label{fig:polygonal_chain}
\end{figure}

\Cref{fig:main_1} illustrates Theorem~\ref{thm:main}. In the left and middle plot we see the zeros, poles and critical curves of a function $f$ for two shifts ${\eta_1}$ and $\eta_2$. Since there is a continuous path from $\eta_1$ to $\eta_2$, which does not cross a caustic of $f$ (see the plot on the right), $f_{\eta_1}$ and $f_{\eta_2}$ have the same number of zeros, and these have the same locations with respect to the critical curves.

Moreover, Theorem~\ref{thm:main} implies that in Theorem~\ref{thm:initial_value} we can replace
$\eta>0$ by any sufficiently large $\widetilde{\eta}\in\C$.

\begin{figure}[h]
\begin{center}
\includegraphics[width=\textwidth]{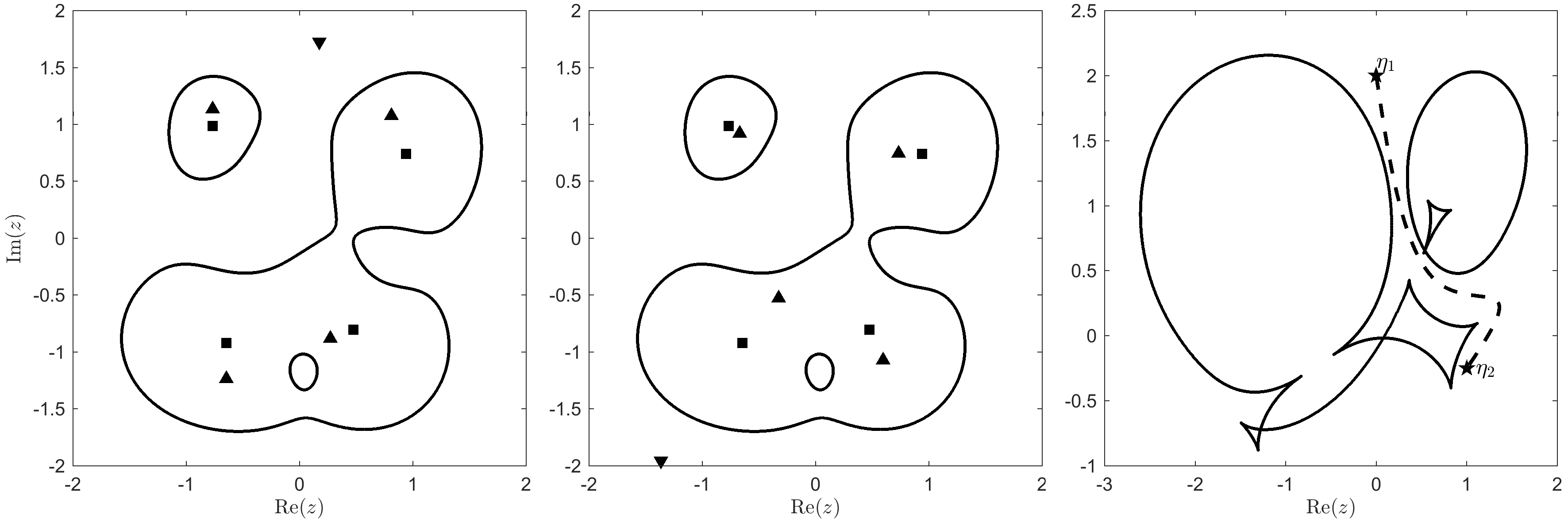}
\caption{Left and middle: Critical curves (solid), sense-preserving zeros
($\blacktriangle$), sense-reversing zeros ($\blacktriangledown$), and poles (\footnotesize$\blacksquare$\normalsize) for $f_{\eta_1}$ (left) and
$f_{\eta_2}$ (middle). Right: Caustics and the corresponding shift (dashed). }\label{fig:main_1}
\end{center}
\end{figure}

\begin{rem}\label{rmk:extremal}
A rational harmoinc function $f$ as in \eqref{eqn:def_f} is called \emph{extremal},
when it has the maximum number of $5\deg(r)-5$ zeros. As mentioned in the Introduction,
an explicit construction of Rhie~\cite{Rhie2003} yields an extremal function $f$ 
with $r$ as in \eqref{eq:r} and $\eta=0$ for each $n=\deg(r)\geq 2$. We then 
have $r=p/q$, where $\deg(p)=n-1$ and $\deg(q)=n$. 
Our Theorems~\ref{thm:initial_value} and~\ref{thm:main} imply that whenever 
$|\eta|$ is 
large enough, the shifted function $f_\eta$ has exactly $n+1$
zeros, namely $n$ zeros close to the poles of $f$, and one zero in $A_\infty$.
Thus, $f_\eta$ has $4n-6$ fewer zeros than the extremal function $f$. 
An example of an extremal rational harmonic function $f$ with $\deg(r)=3$ and hence $10$ zeros is shown in \Cref{fig:exp_mpw}.
In that example a sufficiently large $|\eta|$ leads to a function $f_\eta$ with only $4$ zeros. 
\end{rem}

\section{Crossing a caustic of \texorpdfstring{$f$}{f}}\label{sect:main_section_2}

In this section we will investigate the situation when a constant shift results
in a caustic crossing of a function~$f$ as in \eqref{eqn:def_f}. Let $z_0 \in \CC$
be a critical point of $f$, i.e., $|r'(z_0)| = 1$, and let us define $\eta \defby f(z_0)$,
so that $z_0$ is a singular zero of $f_\eta(z)= f(z) -\eta$. Using the
Taylor series of $r(z)$ at $z=z_0$ and $r(z_0)-\eta=\conj{z}_0$, we then have
\begin{align}\label{eqn:taylor}
f_\eta(z) &= r'(z_0) (z-z_0)+\frac{r''(z_0)}{2} (z-z_0)^2 - (\conj{z} - \conj{z}_0) +
\sum_{k=3}^\infty \frac{r^{(k)}(z_0)}{k!} (z-z_0)^k.
\end{align}
For simplicity of notation we will now assume that
$$z_0=0\quad \mbox{and}\quad r'(0)=1.$$
This assumption amounts to a shift and rotation of the complex plane and hence
it can be made without loss of generality of the results on the zeros of $f_\eta$
that we will derive in the following. Under our assumption we can write \eqref{eqn:taylor} as
\begin{align}
&f_\eta(z) = T(z)+R(z),\quad\mbox{where}\label{eqn:feta1}\\
&T(z) \defby  dz^2+z-\conj{z},\quad d \defby \frac{r''(0)}{2},\quad
R(z) \defby \sum_{k=3}^\infty \frac{r^{(k)}(0)}{k!} z^k.\label{eqn:feta2}
\end{align}
Because of the non-degeneracy assumption on $f$ we have $r''(0)\neq 0$,
and thus $d\neq 0$.

Our strategy in the following is to show that in the neighborhood of $z=0$
the remainder term $R$ is ``small enough'', so that the zeros of $f_\eta$
are close to the zeros of $T$, which can be explicitly analyzed. This
approach is similar in spirit to the perturbation analysis in~\cite{SeteLuceLiesen2015a}. 
Note that since $T$ is a harmonic polynomial
of degree $2$, it has at most $4$ zeros~\cite{KhavinsonSwiatek2003}.

\begin{lem}\label{lem:zerosT}
For a given $\delta\in \R$, let $T_{\delta d}(z) \defby T(z)-\delta d$.
Then all real zeros of $T_{\delta d}$ are given by
\begin{equation}\label{eqn:real_zeros}
z =
\begin{cases}
0, & \mbox{if $\delta=0$,} \\
\pm\sqrt{\delta}, & \mbox{if $\delta>0$,}
\end{cases}
\end{equation}
and all non-real zeros of $T_{\delta d}$ are given by
\begin{equation}\label{eqn:imag_zeros}
z =
\begin{cases}
-d^{-1}\pm i\sqrt{|d^{-2}|-\delta}, & \mbox{if $|d^{-2}|\geq \delta$ and $\re(d^{-1})^2\neq \delta$,} \\
-d^{-1}-i\im(d^{-1}), & \mbox{if $|d^{-2}|\geq \delta$ and $\re(d^{-1})^2= \delta$, } 
\mbox{and $\im(d^{-1})\neq 0$.}
\end{cases}
\end{equation}
In particular, if $z$ is a non-real zero of $T_{\delta d}(z)$,
then $|z| \ge |\re(d^{-1})|>0$.
\end{lem}

\begin{proof}
Let us write $z=x+iy$ and $d^{-1}=\alpha+i\beta$. The equation
$T_{\delta d}(z)=0$ holds if and only if
$$z^2+d^{-1}(z-\conj{z})-\delta = 0.$$
Splitting this equation into its real and imaginary parts gives the two equations
\begin{align}
0 &= y^2+2\beta y-x^2+\delta, \label{eqn:real1}\\
0 &= xy+\alpha y, \label{eqn:real2}
\end{align}
which we need to solve for real $x$ and $y$.

If $y=0$, then \eqref{eqn:real1} implies $x^2=\delta$, where both $x$ and $\delta$ are real.
Thus, all solutions of $T_{\delta d}(z)=0$ with $\im(z)=0$ are given by $z=0$ if $\delta=0$,
and $z=\pm\sqrt{\delta}$ if $\delta>0$. If $\delta<0$, then there exists no real solution.

If $y\neq 0$, then \eqref{eqn:real2} implies $x=-\alpha$, and substituting $x^2=\alpha^2$
in \eqref{eqn:real1} yields
$$y_\pm=-\beta \pm \sqrt{\beta^2+\alpha^2-\delta}.$$
A solution of $T_{\delta d}(z)=0$ with $\im(z)\neq 0$ exists only when
$y_{\pm}\in\R\setminus\{0\}$.

We have $y_{\pm}\in\R$ if and only if $\beta^2+\alpha^2-\delta=|d^{-2}|-\delta\geq 0$.
If this holds, and we additionally have $\alpha^2-\delta=\re(d^{-1})^2-\delta \neq 0$,
then $y_{\pm}\neq 0$, and $T_{\delta d}(z)=0$ has the two non-real solutions
\begin{align*}
z_\pm &= -\alpha+iy_{\pm}=-d^{-1}\pm i\sqrt{|d^{-2}|-\delta}.
\end{align*}
If $\alpha^2-\delta=0$ and $\beta\neq 0$, then $y_+=0$ and $y_{-}=-2\beta\neq 0$, so
that
$$z_{-}=-\alpha+i y_{-} = -d^{-1}-i \im(d^{-1})$$
is the only non-real solution. If $\alpha^2-\delta=0$ and $\beta=0$, then
there exists no non-real solution.
\eop
\end{proof}

\begin{rem}
Lemma~\ref{lem:zerosT} gives a complete characterization of all choices of $\delta\in\R$
that lead to an \emph{extremal} harmonic polynomial $T_{\delta d}$ that has the maximum
number of $4$ zeros. For such a polynomial we need $0<\delta\leq |d^{-2}|$ and
$\delta\neq \re(d^{-1})^2$, and then the $4$ zeros are
$z = \pm \sqrt{\delta}$  and $z = -d^{-1} \pm i\sqrt{ |d^{-1}| - \delta}$.
\end{rem}

Each non-real zero of $T_{\delta d}$ satisfies $|z|\geq |\re(d^{-1})|>0$,
independently of the size of $\delta$. Thus, if $\re(d)\neq 0$, then
for $\delta>0$ small enough,
the only zeros of $T_{\delta d}$ in a (small enough) neighborhood of $z=0$
are the two real zeros $\pm\sqrt{\delta}$. This fact will be very important
in the proof of the following result.

\begin{thm}\label{thm:main_fold}
Let $f$ be as in \eqref{eqn:def_f} with $r'(0)=1$, suppose that the
fold point $\eta := f(0)$ is simple, and let $A_+,A_- \in \mathcal{A}$
be the bordered sets on the critical point~$z=0$. Then there exists a
nonzero $\widetilde{\eta} \in \C$, such that for all $0< \alpha \le 1$
we have
\begin{compactenum}[(i)]
\item{$N(f_{\eta+\alpha\widetilde{\eta}})= N(f_{\eta}) + 1 = N(f_{\eta-\alpha\widetilde{\eta}}) + 2$,}
\item{$N(f_{\eta + \alpha\widetilde{\eta}}; A) =  N(f_\eta; A) = N(f_{\eta - \alpha\widetilde{\eta}}; A)$ for all $A \in \mathcal{A}\setminus\{A_+,A_-\}$,}
\item{$N(f_{\eta + \alpha\widetilde{\eta}}; A_+)  = N(f_{\eta - \alpha\widetilde{\eta}}; A_+)+ 1 $,}
\item{$N(f_{\eta + \alpha\widetilde{\eta}}; A_-) = N(f_{\eta - \alpha\widetilde{\eta}}; A_-)+1$,}
\item{ $N_{s}(f_{\eta + \alpha\widetilde{\eta}}) =
N_{s}(f_{\eta - \alpha\widetilde{\eta}}) =0$, and  $N_{s}(f_\eta) = 1$.}
\end{compactenum}
\end{thm}

\begin{proof}
We will write $f_\eta$ as in \eqref{eqn:feta1}--\eqref{eqn:feta2}, and
for a given $\delta>0$ we will write $T_{\pm \delta d}(z):=T(z)\mp \delta d$.
Since $f(0)$ is a (simple) fold point, we have $\re(d)\neq 0$; see Lemma~\ref{lem:cusp}.

We know that if $\widetilde{\delta} > 0$ is small enough, then there exists
an $\widetilde{\eps}$, depending on $\widetilde{\delta}$
and with $\widetilde{\eps} > \sqrt{\widetilde{\delta}} >0$,
such that the only zeros of $T_{+\widetilde{\delta} d}$ in the open disk $B_{\widetilde{\eps}}(0)$
are the two real zeros $z=\pm \sqrt{\widetilde{\delta}}$. The function $T_{-\widetilde{\delta} d}$
has no zeros in that disk.
Moreover, by shrinking $\widetilde{\eps}$ and $\widetilde{\delta}$
if necessary, we can assume that $f_{\eta\pm\widetilde{\delta} d}$ has no pole in $B_{\widetilde{\eps}}(0)$, since
$|f_{\eta\pm\widetilde{\delta} d}(0)|=|\widetilde{\delta}d|<\infty$.

The orientation of $T_{+\widetilde{\delta} d}$ is determined by its Jacobian
$$J_{T_{+\widetilde{\delta} d}}(z)=|2dz+1|^2-1.$$
Thus, by possibly shrinking $\widetilde{\delta}>0$ once more, we can assume that
$T_{+\widetilde{\delta} d}$ is differently oriented at its two (real) zeros
$z=\pm\sqrt{\widetilde{\delta}}$.

The main idea now is to suitably choose $\eps$ and $\delta$ with
$\eps>\sqrt{\delta}>0$, by possibly further shrinking the values $\widetilde{\eps}$ and
$\widetilde{\delta}$ obtained above, so that we can successfully apply Theorem~\ref{thm:Rouche}
to $f_{\eta\pm \delta d}$ and $T_{\pm\delta d}$ on the closed Jordan curves
\begin{align}\label{eq:Gamma+-}
\Gamma^+ \defby \left(B^+, B_\eps(0) \cap \mathcal{C}\right) \quad \text { and } \quad \Gamma^- \defby \left(B^-, B_\eps(0) \cap \mathcal{C}\right),
\end{align}
where $B^+ \defby \partial B_\eps(0) \cap \Omega_+$ and $B^- \defby \partial B_\eps(0) \cap \Omega_-$.
Thus, we have to verify that
\begin{align}\label{eq:rouche_tmp}
|f_{\eta\pm \delta d}(z) - T_{\pm \delta d}(z)| = |R(z)|
< |f_{\eta\pm \delta d}(z)| +  |T_{\pm \delta d}(z)|
\end{align}
for all $z \in  \partial B_\eps(0) \cup \left( B_\eps(0) \cap \mathcal{C}\right)$.

The following argument is quite technical since the
constant shift $\pm\delta d$ and the radius $\epsilon$ influence each
other. 

In the neighborhood of $z=0$ we have
$$|T(z)| \in \mathcal{O}(|z|) \quad\mbox{and}\quad |R(z)| \in \mathcal{O}(|z|^3),$$
and consequently
\begin{align*}
|R(z)| < |T(z)| \quad\mbox{in $B_{\eps^*}(0)$ for a sufficiently small $\eps^* > 0$.}
\end{align*}
We can assume that $0<\eps^*\leq \widetilde{\eps}$, and we now have to find a corresponding
$\delta^*>0$. To this end we define
\begin{align*}
\eps(\delta) \defby \max \{ 0 \le \varepsilon \le \varepsilon^* : |R(z)| < |T(z)| - \delta |d| 
\text{ for all } z \in \partial B_\varepsilon(0) \},
\end{align*}
which is a continuous function of the real variable $\delta\geq 0$. We have 
$\eps(0)=\eps^*>0$ and
$\lim_{\delta\rightarrow\infty}\eps(\delta)=0$, where $B_0(0)=\emptyset$.
Thus, for every continuous and strictly monotonically increasing function
$$\psi: [0,\infty) \rightarrow [0,\infty)\quad\mbox{with}\quad \psi(0) = 0,$$
there exists a $\delta^* > 0$, such that $\eps(\delta^*) > \psi(\delta^*)$.
Using the function
$\psi(t) \defby \sqrt{t}$ yields parameters $\delta^*$ and $\eps(\delta^*)$ with
$\widetilde{\eps}>\eps(\delta^*)>\sqrt{\delta^*}>0$, so that only
the two real zeros $z_\pm=\pm\sqrt{\delta^*}$ of $T_{+\delta^*d}$ lie in the
disk $B_{\eps(\delta^*)}(0)$, and $T_{+\delta^*d}$ is differently oriented
at these zeros.

For all $z \in \partial B_{\eps(\delta^*)}(0)$ we immediately obtain
\begin{align*}
|f_{\eta\pm \delta^* d}(z) - T_{\pm \delta^* d}(z)| = |R(z)|
< |T(z)| - \delta^*|d|\le |f_{\eta\pm \delta^* d}(z)| +  |T_{\pm \delta^* d}(z)|.
\end{align*}

We also have to verify inequality \eqref{eq:rouche_tmp} on $B_{\eps(\delta^*)}(0) \cap \CC$.
Using \eqref{eq:tangent} with our assumptions $z_0=0$, $r'(0)=1$, and $d=r''(0)/2$, we see that
this curve is given by
$$g(t) = h t + \mathcal{O}(|t|^2),\quad
h \defby \frac{i\conj{d}}{|d|}.$$
A straightforward computation shows that
\begin{align*}
T(ht) =
-\conj{d}t^2  + 2i\frac{\re(d)}{|d|} t.
\end{align*}
For $\eps^*>0$ sufficiently small and $|t| \le \sqrt{\delta^*/2}$ we obtain,
by restricting to the real part,
\begin{align*}
\left|T_{\pm \delta^* d}(h t)\right| &=
\left|-\conj{d}t^2   + 2i\frac{\re(d)}{|d|} t \mp \delta^* d \right| \ge |\re(d)|\cdot|t^2 \pm \delta^*| \\ &\ge  \frac{|\re(d)|\delta^*}{2}
> c_1|t|^3 \ge \left|R(h t)\right|,
\end{align*}
where $c_1>0$ is a real constant, and we have used that $\re(d)\neq 0$.

On the other hand, for $\eps^*>0$ sufficiently small and$\sqrt{\delta^*/2} \le |t| \le \eps(\delta^*)$
we obtain, by restricting to the imaginary part,
\begin{align*}
\left|T_{\pm \delta^* d}(h t)\right| &=\left|-\conj{d}t^2   + 2i\frac{\re(d)}{|d|} t \mp \delta^* d \right| \ge \left|\im(d)(t^2 \mp \delta^*) + 2\frac{\re(d)}{|d|}t \right| \\
&\ge c_2 \sqrt{\frac{\delta^*}{2}} > c_3|t|^3 \ge \left|R(h t)\right|,
\end{align*}
where $c_2,c_3>0$ are real constants. Note that in order to obtain the
second inequality, it is again necessary that $\re(d)\neq 0$.

Together we have
\begin{align}\label{eq:tmp_main_2}
\left|T_{\pm \delta^* d}(h t)\right| >
\left|R(h t)\right|\quad\mbox{for}\quad |t| < \eps(\delta^*).
\end{align}
Clearly, if we do the same computations with
$g(t)=ht+\mathcal{O}(|t|^2)$ instead of $ht$, we obtain
the same estimate as in \eqref{eq:tmp_main_2} for a possibly smaller $\eps^*>0$. Hence,
\begin{align*}
\left|f_{\eta\pm \delta^*d}(g(t))-T_{\pm \delta^*d}(g(t))\right|
&=
\left|R(g(t))\right| < \left|T_{\pm \delta^*d}(g(t))\right| \\
&\le \left|T_{\pm \delta^*d}(g(t))\right| +
\left|f_{\eta\pm \delta^*d}(g(t))\right|
\end{align*}
holds for all $t$ with $|g(t)| \le \eps(\delta^*)$. Consequently, \eqref{eq:rouche_tmp} is fulfilled for all $z \in B_{\eps(\delta^*)}(0) \cap \CC$.

In summary, we can apply Theorem~\ref{thm:Rouche} on $\Gamma^+$ and $\Gamma^-$, see \eqref{eq:Gamma+-}. With Corollary~\ref{cor:zero_counting} this yields
\begin{align*}
\begin{aligned}
N(f_{\eta+\delta^* d}; \intx(\Gamma^+))&=V(f_{\eta+\delta^* d}; \Gamma^+) =
V(T_{+\delta^* d}; \Gamma^+) = 1,\\
N(f_{\eta+\delta^* d}; \intx(\Gamma^-))&=-V(f_{\eta+\delta^* d}; \Gamma^-) =
-V(T_{+\delta^* d}; \Gamma^-) = 1,\\
N(f_{\eta-\delta^* d}; \intx(\Gamma^+))&=V(f_{\eta-\delta^* d}; \Gamma^+) =
V(T_{-\delta^* d}; \Gamma^+)= 0,\\
N(f_{\eta-\delta^* d}; \intx(\Gamma^-))&=-V(f_{\eta-\delta^* d}; \Gamma^-) =
-V(T_{-\delta^* d}; \Gamma^-) = 0.
\end{aligned}
\end{align*}
%
%
Using Lemma~\ref{lem:main} and Theorem~\ref{thm:small_perturbation} (again for possibly smaller $\eps^*>0$), we see that the assertions (iii) and (iv) are fulfilled for
$\alpha=1$ and $\widetilde{\eta} := \delta^* d$. The same argument
as in the proof of Lemma~\ref{lem:main_2} gives assertion (ii), and therefore
also (i) and (v) follow (all for $\alpha=1$).

Finally, the assertions (i)--(v) hold for all $0<\alpha\leq 1$, since
for sufficiently small $\delta > 0$ the line between $+\delta d$
and $-\delta d$ contains only a single caustic point of $f$.
\eop
\end{proof}

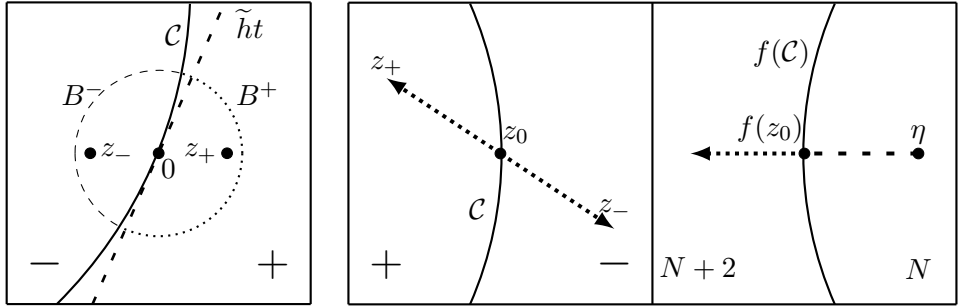
\begin{figure}[!ht]
\centering
\subfigure[Construction in the proof of  Theorem~\ref{thm:main_fold}.]{\label{fig:construction} \begin{tikzpicture}

\draw [thick] plot [smooth] coordinates {(0,0) (4,0)};
\draw [thick] plot [smooth] coordinates {(0,0) (0,4)};
\draw [thick] plot [smooth] coordinates {(4,0) (4,4)};
\draw [thick] plot [smooth] coordinates {(0,4) (4,4)};

\draw[color=black] (1.9,1.8) node[right] {$0$};
\draw[color=black] (1.1,2) node[right] {$z_-$};
\draw[color=black] (2.9,2) node[left] {$z_+$};

\draw[color=black] (3.5,0.2) node[above] {\LARGE $+$};
\draw[color=black] (.5,0.2) node[above] {\LARGE $-$};

\draw[color=black] (3.3,2.5) node[above] { $B^+$};
\draw[color=black] (1,2.5) node[above] { $B^-$};


\draw[color=black] (2.2,3.3) node[above] {$\mathcal C$};

\draw [dotted,thick] (3.1,2) arc (0:75:1.1);
\draw [dotted,thick] (3.1,2) arc (0:-117:1.1);
\draw [dashed] (0.9,2) arc (180:243:1.1);
\draw [dashed] (0.9,2) arc (180:75:1.1);

\draw [thick] (.66,0) arc (-45:-2.5:6);

\draw[fill=black](2,2)circle(2pt);
\draw[fill=black](1.1,2)circle(2pt);
\draw[fill=black](2.9,2)circle(2pt);
\draw[color=black] (3.2,3.4) node[above] {$\widetilde{h}t $};
\draw [line width=1pt,loosely dashed] (1.13,0) to (2.87,4);

\end{tikzpicture}}
\hfill
\subfigure[Crossing a caustic at a fold point (cf. {\cite[Figures 9.2 and 9.3]{PettersLevineWambsganss2001}}).]{\label{fig:schema_fold}\begin{tikzpicture}
\usetikzlibrary{positioning,arrows}
\tikzset{arrow/.style={-latex}}
\draw [thick] plot [smooth] coordinates {(-4,-2) (4,-2)};
\draw [thick] plot [smooth] coordinates {(-4,-2) (-4,2)};
\draw [thick] plot [smooth] coordinates {(4,-2) (4,2)};
\draw [thick] plot [smooth] coordinates {(-4,2) (4,2)};
\draw [thick] plot [smooth] coordinates {(0,-2) (0,2)};

\draw[color=black] (-2.3,-1) node[above] {$\mathcal C$};
\draw[color=black] (-3.5,-1.8) node[above] {\LARGE $+$};
\draw[color=black] (-.5,-1.8) node[above] {\LARGE $-$};

\draw[color=black] (-.5,-1) node[above] {$z_-$};
\draw [line width=1.5pt,arrow, dotted] (-2,0) to (-.5,-1);
\draw[color=black] (-3.5,.9) node[above] {$z_+$};
\draw [line width=1.5pt,arrow, dotted] (-2,0) to (-3.5,1);
\draw[fill=black](-2,0)circle(2pt);
\draw[color=black] (-1.8,0) node[above] {$z_0$};

\draw[color=black] (1.7,1) node[above] {$f(\mathcal C)$ };
\draw[color=black] (3.5,-1.8) node[above] {$N$};
\draw[color=black] (.6,-1.8) node[above] {$N+2$};

\draw [line width=1.5pt,arrow, dotted] (2,0) to (.5,0);
\draw [line width=1.5pt,loosely dashed] (3.5,0) to (2,0);
\draw[fill=black](3.5,0)circle(2pt);
\draw[color=black] (3.5,0) node[above] {$\eta$};
\draw[fill=black](2,0)circle(2pt);
\draw[color=black] (1.55,0) node[above] {$f(z_0)$};

\draw [thick] (-2.4,-2) arc (-23.5:23.5:5);
\draw [thick] (2.4,-2) arc (203.5:156.5:5);

\end{tikzpicture}}
\caption{Local behavior near fold points.}
\end{figure}

While we have formulated Theorem~\ref{thm:main_fold} for the critical point $z_0=0$
and for $r'(0)=1$, it is clear that the result holds for any $z_0\in\CC$
and the corresponding value $r'(z_0)=e^{i\varphi}$, as long as $\eta=f(z_0)$ is
a simple fold point. For a \emph{multiple} fold point $\eta$, the set of
corresponding critical points $f^{-1}(\eta)$ contains more than one element,
and then the effect of Theorem~\ref{thm:main_fold} happens simultaneously at each of
these critical points. An example can be seen in Figure~\ref{fig:exp_mpw}, where 
one of the caustics has~$9$ double fold points. When the caustic is crossed at 
one of these points in a suitable direction, the number of zeros of the shifted functions changes by~$4$. 

In the proof of Theorem~\ref{thm:main_fold}, the crossing of the caustic at a (simple)
fold point was done in the direction $d$, i.e., we considered a shift on
the line from $\eta -\delta^*d$ to $\eta + \delta^*d$. Using Theorem~\ref{thm:main}, we
easily see that crossing the caustic in any other direction yields the same
conclusion on the zeros of the shifted functions.

An illustration of the \emph{local} behavior near a fold point is given
in \Cref{fig:schema_fold}.
We shift the constant term $\eta$ along the dotted line. Coming from the right, the
function $f_\eta$ has no zero close to the critical point $z_0$. For $\eta=f(z_0)$
there is exactly one (singular) zero of $f_\eta$, and after $\eta$ crossed the
caustic of $f$, a pair of differently oriented zeros of $f_\eta$ appears.

An illustration of the \emph{global} effect of caustic crossings
is shown in \Cref{fig:main_2}.
The plots on the left and in the middle show the critical curves, zeros, and poles
of two functions $f_{\eta_1}$ and $f_{\eta_2}$. On the right we plot the caustics
and one possible path from $\eta_1$ to $\eta_2$. On every path from $\eta_1$ to
$\eta_2$ we have at least three caustic crossings. With each crossing a pair
of zeros in the neighborhood of the corresponding critical point appears or
disappears. In this example we have a net gain of $2$~zeros when traveling
from $\eta_1$ to $\eta_2$, and a net loss of $2$~zeros when traveling in the
other direction.

The effect of $2$~additional or $2$~disappearing zeros is determined by the curvature
of the caustic, which is given by the coefficient $-\conj{d}$ of the quadratic term of $T$,
i.e., the caustic is locally a parabola. We have $2$~additional zeros in case of crossing
the caustic coming from the ``open side'' of the parabola, and $2$~disappearing zeros
coming from the other side; see \Cref{fig:main_2}, \Cref{fig:crossing}, and the 
examples in Section~\ref{sect:examples}.

\begin{figure}[h]
\begin{center}
\includegraphics[width=1\textwidth]{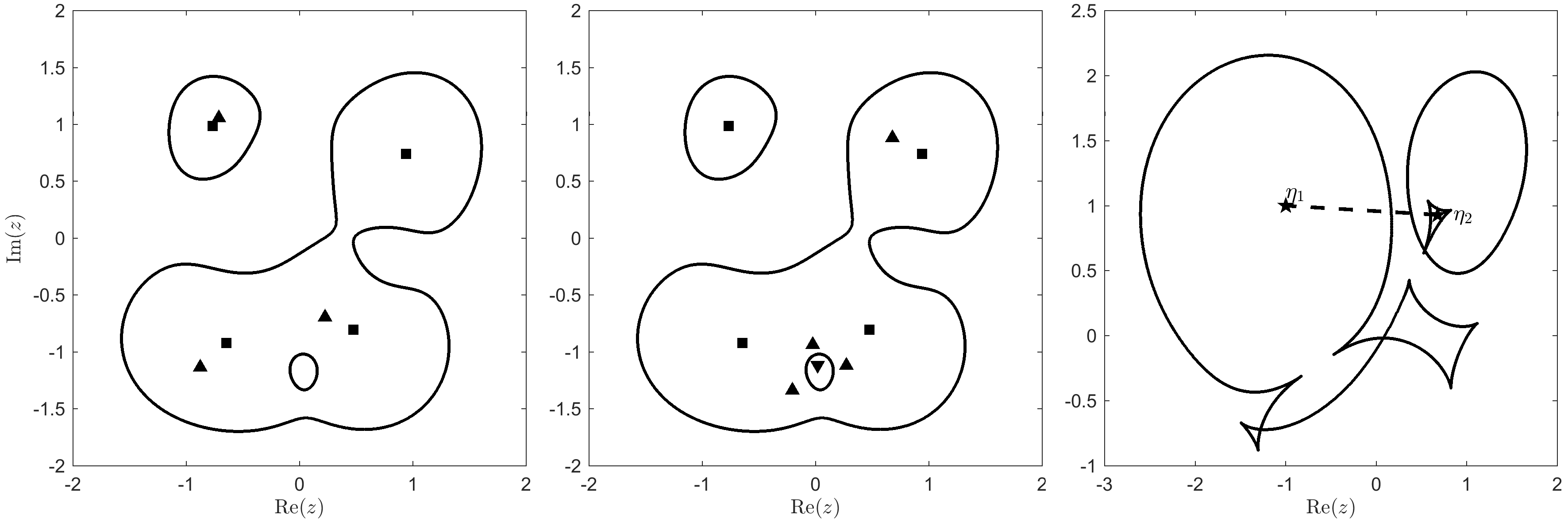}
\caption{Left and middle: Critical curves (solid), sense-preserving zeros
($\blacktriangle$), sense-reversing zeros ($\blacktriangledown$), and poles (\footnotesize$\blacksquare$\normalsize) for $f_{\eta_1}$ (left) and
$f_{\eta_2}$ (middle). Right: Caustics and the corresponding shift (dashed). }\label{fig:main_2}
\end{center}
\end{figure}

\begin{figure}[!ht]
\centering
\subfigure[Relation between caustic curvature and the effect of constant shifts.]{\label{fig:curvature}
\begin{tikzpicture}
\usetikzlibrary{positioning,arrows}
\tikzset{arrow/.style={-latex}}
\draw [thick] plot [smooth] coordinates {(0,0) (4,0)};
\draw [thick] plot [smooth] coordinates {(0,0) (0,4)};
\draw [thick] plot [smooth] coordinates {(4,0) (4,4)};
\draw [thick] plot [smooth] coordinates {(0,4) (4,4)};
\draw[color=black] (2,1.8) node[right] {$\eta$};
\draw[color=black] (1,2.4) node[below] {$-\conj{d}$};
\draw[color=black] (3,2.5) node[below] {$d$};
\draw[color=black] (.7,1.15) node[above] {$\eta -\delta^*d$};
\draw[color=black] (3.3,2.65) node[above] {$\eta + \delta^*d$};
\draw[color=black] (3,0.2) node[above] {$N+2$};
\draw[color=black] (.5,0.2) node[above] {$N$};
\draw[color=black] (1.2,3.3) node[above] {$T(\widetilde{h}t)$};
\draw [thick] (2,2) arc (0:19.5:6);
\draw [thick] (2,2) arc (0:-53:2.5);
\draw [arrow, dotted, thick] (2,2) to (3,2.5);
\draw [arrow, dotted, thick] (2,2) to (1,2.5);
\draw[fill=black](2,2)circle(2pt);
\draw[fill=black](3.5,2.75)circle(2pt);
\draw[fill=black](.5,1.25)circle(2pt);
\draw [line width=1pt,loosely dashed] (2,0) to (2,4);
\end{tikzpicture}}
\subfigure[Avoiding a cusp crossing.]{\label{fig:crossing}
\usetikzlibrary{positioning,arrows}
\begin{tikzpicture}
\draw [thick] plot [smooth] coordinates {(3,0) (7,0)};
\draw [thick] plot [smooth] coordinates {(3,4) (3,0)};
\draw [thick] plot [smooth] coordinates {(3,4) (7,4)};
\draw [thick] plot [smooth] coordinates {(7,4) (7,0)};
\draw [thick] plot [smooth] coordinates {(3,1.625) (4,1.75) (5,2) (6,2.5) (6.5,3)};
\draw [thick] plot [smooth] coordinates {(6.25,0) (6,0.5) (6,1.5)(6.3,2.5) (6.5,3)};
\draw [dashed] plot  coordinates {(5.45,1)(6.75,3.5)};
\draw [densely dotted] plot  coordinates {(5.45,1)(3.5,1)(3.5,2.6)(6.75,3.5)};
\draw[fill=black](6.75,3.5)circle(2pt);
\draw[color=black] (6.75,3.6) node[left] {$\eta_1$};
\draw[fill=black](5.45,1)circle(2pt);
\draw[color=black] (5.45,1) node[right] {$\eta_2$};
\draw[fill=black](3.5,2.6)circle(2pt);
\draw[color=black] (3.25,2.4) node[right] {$\eta_{1}'$};
\draw[fill=black](3.5,1)circle(2pt);
\draw[color=black] (3.25,.75) node[right] {$\eta_{2}'$};
\draw[fill=black](3.5,1.7)circle(2pt);
\draw[color=black] (3.5,1.9) node[right] {$\eta$};
\draw[color=black] (6.1,.3) node[left] {caustic};
\end{tikzpicture}}
\caption{}
\end{figure}
We are able to ``simulate'' the crossing of a cusp point using Theorem~\ref{thm:main} and Theorem~\ref{thm:main_fold}; see \Cref{fig:crossing}.
However, we also would like to give a local characterization of a cusp crossing. An important ingredient is the following result of Sheil-Small;
see~\cite[Theorem~14]{Wilmshurst1994}.

\begin{prop}\label{prop:sheil_small}
If $g$ is an analytic function in the convex domain $D$ with $|g'(z)|<1$
in $D$, then $g(z)-\conj{z}$ is univalent in $D$.
\end{prop}

If $g$ is analytic and $|g'(z)|<1$ in a star domain $D$ with base point $z_0$, then we can apply this proposition on the lines from $z_0$ to any point of $D$, which implies that $g(z)-\conj{z}$ attends the value $g(z_0)-\conj{z}_0$ exactly once in $D$. This fact will be used in the proof of the next theorem. 

\begin{thm}\label{thm:main_cusp}
Let $f$ be as in \eqref{eqn:def_f} with $r'(0)=1$, suppose that the
cusp point $\eta \defby f(0)$ is simple, and let $A_+,A_- \in \mathcal{A}$
be the bordered sets on the critical point~$z = 0$. Then there exist a
nonzero $\widetilde{\eta} \in \C$ and $b_+, b_- \in \{0,1\}$ with $b_+ + b_- = 1$, such that for all $0< \alpha \le 1$
we have
\begin{compactenum}[(i)]
\item{$N(f_{\eta + \alpha\widetilde{\eta}}; A_+)   = N(f_\eta; A_+) - b_+ = N(f_{\eta - \alpha\widetilde{\eta}}; A_+) + 1$,}
\item{$N(f_{\eta + \alpha\widetilde{\eta}}; A_-)  = N(f_\eta; A_-) - b_-  = N(f_{\eta - \alpha\widetilde{\eta}}; A_-) + 1$,}
\item{$N_s(f_{\eta + \alpha\widetilde{\eta}})  = N_s(f_{\eta - \alpha\widetilde{\eta}})   = 0$, and $N_s(f_\eta) = 1$.}
\end{compactenum}
\end{thm}

\begin{proof}
The equalities
\begin{align*}
N(f_{\eta + \alpha\widetilde{\eta}}; A_+)  &= N(f_{\eta - \alpha\widetilde{\eta}}; A_+) + 1,\\
N(f_{\eta + \alpha\widetilde{\eta}}; A_-)  &= N(f_{\eta - \alpha\widetilde{\eta}}; A_-) + 1,\\
N_s(f_{\eta + \alpha\widetilde{\eta}})  &= N_s(f_{\eta - \alpha\widetilde{\eta}})   = 0,
\end{align*}
already follow from Theorem~\ref{thm:main_fold} and Lemma~\ref{lem:main_2}; see \Cref{fig:crossing}.
In order to show the remaining assertions we now investigate, as in the proof of Theorem~\ref{thm:main_fold}, the functions
$f_{\eta \pm \delta d}$ and $T_{\pm\delta d}$. Since we are in the cusp case, we have
$\re(d)=0$ (see Lemma~\ref{lem:cusp}), and hence the non-real zeros of $T_{\pm\delta d}$ come
into play.

From Lemma~\ref{lem:zerosT} we know that for all $0<\delta< |d^{-2}|$,
the function $T_{+\delta d}$ has the two real zeros $z=\pm\sqrt{\delta}$, while
$T_{-\delta d}$ has no real zeros. Moreover, $T_{+\delta d}$ has the two purely
imaginary zeros
\begin{align}\label{eq:imag_zeros1}
z=-d^{-1}\pm i\sqrt{|d^{-2}|-\delta},
\end{align}
and $T_{-\delta d}$ has the two purely imaginary zeros
\begin{align}\label{eq:imag_zeros2}
z=-d^{-1}\pm i\sqrt{|d^{-2}|+\delta}.
\end{align}
Only one of the two zeros in \eqref{eq:imag_zeros1} and in \eqref{eq:imag_zeros2}
is sufficiently close to $z=0$, and the sign of $\im(d^{-1})$ determines which one
it is: If $\im(d^{-1}) > 0$, then the zero of interest of $T_{\pm\delta d}$ is
\begin{align*}
&z_{\pm} \defby -d^{-1} + i\sqrt{|d^{-2}|\mp \delta}
\end{align*}
since then $|z_{\pm}|< \sqrt{\delta}$, while the other zero satisfies
\begin{align*}
|-d^{-1} - i\sqrt{|d^{-2}|\mp \delta}|\ge |d^{-1}|>\sqrt{\delta}.
\end{align*}
From
\begin{align*}
|\partial_z T_{\pm\delta d}(z_{\pm})| &= |2dz_{\pm} + 1| = |2d(-d^{-1}+i\sqrt{|d^{-2}|\mp\delta}) + 1| \\
&= |-1 + 2di\sqrt{|d^{-2}|\mp\delta}| = |-1 + 2\sqrt{1 \mp \delta |d|^2}|
\end{align*}
we see that $z_{+}$ is a sense-reversing zero of $T_{+\delta d}$, and
$z_{-}$ is a sense-preserving zero of $T_{-\delta d}$. For $\im(d^{-1}) < 0$
we get an analogous result, but then the zeros in \eqref{eq:imag_zeros1} and \eqref{eq:imag_zeros2} change their
roles, i.e., $z=-d^{-1} - i\sqrt{|d^{-2}|\mp \delta}$ is close to zero, and
$z=-d^{-1} + i\sqrt{|d^{-2}|\mp \delta}$ is bounded away from zero.

We will now show that the zero $z_{\pm}$ of $T_{\pm\delta d}$
corresponds to a zero of $f_{\eta \pm \delta d}$ by applying Theorem~\ref{thm:Rouche}
on $\partial B_{\widetilde{\eps}}(z_{\pm})$ for an appropriately chosen
$\widetilde{\eps}>0$. For each $\varphi\in [0,2\pi)$ we have
\begin{align*}
T_{\pm \delta d}(z_{\pm} + \widetilde{\eps}e^{i\varphi})
&= \widetilde{\eps}\left( 2dz_{\pm}e^{i\varphi} + d\widetilde{\eps} e^{2i\varphi} + 2i\im(e^{i\varphi})\right).
\end{align*}
For a sufficiently small $\delta>0$, which determines $z_{\pm}$,
we now set $\widetilde{\eps} \defby |z_{\pm}|$, and we assume that
$2|d|\widetilde{\eps} \le 1$. Then
\begin{align*}
\left|T_{\pm \delta d}(z_{\pm} + \widetilde{\eps}e^{i\varphi})\right|
&\ge \widetilde{\eps}\left( 2\left|\pm |d| \widetilde{\eps}e^{i\varphi} + i\im(e^{i\varphi})\right| - |d|\widetilde{\eps} \right) \\
&\ge \widetilde{\eps}\left( 2\left|\pm |d| \widetilde{\eps}\re(e^{i\varphi}) + (1\pm|d|\widetilde{\eps})i\im(e^{i\varphi})\right| - |d|\widetilde{\eps} \right) \\
&\ge \widetilde{\eps}\left( 2\left|\pm |d| \widetilde{\eps}\re(e^{i\varphi}) + |d|\widetilde{\eps}i\im(e^{i\varphi})\right| - |d|\widetilde{\eps} \right) \\
&\ge \widetilde{\eps}\left(2|d|\widetilde{\eps}  - |d|\widetilde{\eps}\right) = |d|\widetilde{\eps}^2 \\
&> c|z_{\im}+ \widetilde{\eps}e^{i\varphi}|^3 \ge |R(z_{\im}+ \widetilde{\eps}e^{i\varphi})|,
\end{align*}
for some constant $c>0$. Thus, we have
\begin{align*}
|f_{\eta\pm\delta d}(z) - T_{\pm \delta d}(z)| = |R(z)| < |T_{\pm \delta d}(z)| + |f_{\eta\pm\delta d}(z)| \quad \text{ for all } z \in \partial B_{\widetilde{\eps}}(z_{\pm}).
\end{align*}
Using Theorem~\ref{thm:Rouche} gives
$$V(f_{\eta \pm \delta d};\partial B_{\widetilde{\eps}}(z_{\pm})) = V(T_{\pm \delta d};
\partial B_{\widetilde{\eps}}(z_{\pm})),$$
and, as a consequence,
$$N(f_{\eta\pm \delta d}; B_{\widetilde{\eps}}(z_{\pm})) = N(T_{\pm \delta d}; B_{\widetilde{\eps}}(z_{\pm})) = 1.$$
In the following we denote by $\widetilde{z}_{\pm}$ the zero of $f_{\eta\pm\delta d}$ corresponding to the zero $z_{\pm}$ of $T_{\pm \delta d}$. By construction, $\widetilde{z}_{+}$ is a sense-reversing zero of $f_{\eta+\delta^*d}$ and $\widetilde{z}_{-}$ is a sense-preserving zero of $f_{\eta-\delta^*d}$.

We now construct $\eps, \delta > 0$ such that we can apply Theorem~\ref{thm:Rouche} on $\partial B_\eps(0)$ and the zeros $\widetilde{z}_\pm$ of $f_{\eta \pm \delta d}$ are in $B_\eps(0)$. Let $\eps > 0$ be such that $f_\eta(z) \neq 0 $ for all $z \in B_\eps(0) \setminus\{0\}$, and $f_\eta$ has no poles in $B_\eps(0)$. Furthermore we define 
\begin{align*}
\widetilde{\delta} \defby \min\left\{\frac{|f_\eta(z)|}{|d|} : z \in \partial B_\eps(0)\right\} \quad \text{ and } \quad \delta \defby \frac{1}{2}\min\{\widetilde{\delta},\eps^2\}.
\end{align*}
Hence, we have
$$
|f_\eta(z) - f_{\eta \pm \delta d}(z)| = \delta|d| < |f_\eta(z)| + |f_{\eta\pm \delta d}(z)| \quad \text{ for all } z \in \partial B_\eps(0).
$$
With Theorem~\ref{thm:Rouche} we get 
\begin{align}\label{eq:tmp_cusp}
V(f_{\eta\pm\delta d};\partial B_\eps(0)) = V(f_{\eta};\partial B_\eps(0)) \quad \text{ and } \quad \widetilde{z}_\pm \in B_\eps(0).
\end{align}

Now we look at the number of zeros of $f_{\eta\pm \delta d}$. 
Since $B_{\eps}(0) \cap \Omega_-$ is a star domain with base point
$\widetilde{z}_{+}$, the function $f_{\eta +\delta d}$
has no other zero than $\widetilde{z}_{+}$ in this domain; see Proposition~\ref{prop:sheil_small} and its discussion. Consequently, the function
$f_{\eta-\delta d}$, which results from crossing the caustic of $f$ through the
cusp point $\eta$, has either no ($b_- = 0$) or two ($b_-=1$) zeros in
$B_{\eps}(0) \cap \Omega_-$; cf. Theorem~\ref{thm:main_fold}. Furthermore, because of \eqref{eq:tmp_cusp}, the function
$f_{\eta-\delta d}$ has either two ($b_+ = 1$) or no ($b_+ = 0$) fewer zeros
than $f_{\eta +\delta d}$ in $B_{\eps}(0) \cap \Omega_+$. Together this
implies the remaining equalities in (i) and (ii) for $\alpha = 1$.

Finally, the assertions (i)--(v) hold for all $0<\alpha\leq 1$, since
for sufficiently small $\delta > 0$ the line between $+\delta d$
and $-\delta d$ contains only a single caustic point of $f$.
\eop
\end{proof}

It is clear that Theorem~\ref{thm:main_cusp} holds for an arbitrary $z_0 \in \mathcal C$, as long as $f(z_0)$ is a simple cusp point. For multiple points the effect happens again simultaneously at all corresponding critical points.

A cusp crossing is illustrated in \Cref{fig:schema_cusp}. We shortly describe the positive case. The constant term $\eta$ is shifted along the dotted line. Coming from the right the function has only one sense-preserving zero close to $z_0$. When $\eta$ reaches the caustic, the unique zero becomes singular. When $\eta$ crosses the caustic, the initial zero crosses the critical curve and thus changes the orientation, i.e., it is now sense-reversing. Furthermore an additional pair of sense-preserving zeros appears. Hence we have three zeros after the caustic crossing. The same happens in the negative case with the reverse orientation.

Finally, it is worth to point out that our results yield a characterization of the Poincar\'e index
of a singular zero. 

\begin{cor}\label{cor:singular}
If $f$ is as in \eqref{eqn:def_f}, and $z_0\in\C$ is a singular zero of $f$, then 
\begin{align*}
\ind(f;z_0) = \begin{cases}
0, \quad &\text{ if } \,\, 0 \,\, \text{ is a fold point,} \\
\pm 1, \quad &\text{ if } \,\, 0 \,\, \text{ is a cusp point.}
\end{cases}
\end{align*}
\end{cor}

\begin{proof}
Let $z_0 \in \C$ be a singular zero of $f$ and choose some $\eps > 0$, such
that $f$ has no other zero in $\overline{B_\eps(z_0)}$. Using the same idea as 
in the proof of Theorem~\ref{thm:small_perturbation}, we define
$$\delta \defby \min\{|f(z)| : z \in \partial B_{\eps}(z_0)\}.$$
Then for each $\eta \in B_\delta(0)$ we have
$$|f(z) - f_\eta(z)| = |\eta| < \delta \le |f(z)| \le |f(z)| + |f_\eta(z)|$$
for all $z \in \partial B_{\eps}(z_0)$, and Theorem~\ref{thm:Rouche} implies that
$$V(f;\partial B_\eps(z_0)) = V(f_{\eta};\partial B_\eps(z_0)).$$
The assertion now follows from the proofs of Theorems \ref{thm:main_fold} (fold case) and \ref{thm:main_cusp} (cusp case); see also the Figures \ref{fig:schema_fold} 
and \ref{fig:schema_cusp}.
\eop
\end{proof}

The cases $+1$ or $-1$, i.e., for positive or negative cusps, are determined by $b_+$ and $b_-$ in Theorem~\ref{thm:main_cusp}.
We see that the Poincar\'e index of a singular zero $z_0$ is the sum of the Poincar\'e indices 
of the regular zeros merging in $z_0$. Recently the Poincar\'e index of singular zeros of harmonic functions $h(z) - \conj{z}$, 
with a general analytic function $h$, were studied in \cite{LuceSete2017} using the power series of $h$. However, a characterization whether the index is $+1$ or $-1$ in the cusp case is pointed out as future work.

\begin{figure}[!ht]
\centering
\subfigure[Crossing a positive cusp (cf. {\cite[Figure 9.7]{PettersLevineWambsganss2001}}).]{
\begin{tikzpicture}
\usetikzlibrary{positioning,arrows}
\tikzset{arrow/.style={-latex}}
\draw [thick] plot [smooth] coordinates {(-4,-2) (4,-2)};
\draw [thick] plot [smooth] coordinates {(-4,-2) (-4,2)};
\draw [thick] plot [smooth] coordinates {(4,-2) (4,2)};
\draw [thick] plot [smooth] coordinates {(-4,2) (4,2)};
\draw [thick] plot [smooth] coordinates {(0,-2) (0,2)};

\draw[color=black] (-2.4,-1.25) node[below] {$\CC$};
\draw[color=black] (-3.7,2) node[below] {\LARGE $-$};
\draw[color=black] (-.5,-2) node[above] {\LARGE $+$};
\draw[color=black] (-1.95,1.4) node[above] {$z_1$};
\draw [line width=1.5pt,arrow, dotted] (-1.95,0) to (-1.95,1.5);
\draw[color=black] (-1.95,-1.4) node[below] {$z_2$};
\draw [line width=1.5pt,arrow, dotted] (-1.95,0) to (-1.95,-1.5);
\draw [line width=1.5pt,arrow, dotted] (-2,0) to (-3.5,0);
\draw [line width=1.5pt, loosely dashed] (-2,0) to (-.3,0);
\draw[color=black] (-3.4,0) node[left] {$z$};
\draw[color=black] (-2.25,0) node[above] {$z_0$};
\draw[fill=black](-1.95,0)circle(2pt);
\draw [thick] (-2.4,-2) arc (-23.5:23.5:5);

\draw[fill=black](3.5,0)circle(2pt);
\draw[fill=black](2.5,0)circle(2pt);
\draw[color=black] (2.5,0) node[above] {$f(z_0)$};
\draw[color=black] (3.5,0) node[above] {$\eta$};
\draw[color=black] (3.5,2) node[below] {$N$};
\draw[color=black] (.6,-.7) node[above] {$N+2$};
\draw [line width=1.5pt, loosely dashed] (2.5,0) to (3.5,0);
\draw [line width=1.5pt,arrow, dotted] (2.5,0) to (.5,0);
\draw [thick] (2.5,0) arc (270:180:2);
\draw [thick] (2.5,0) arc (90:180:2);
\draw[color=black] (0.5,-1.5) node[right] {$f(\CC)$};
\end{tikzpicture}}
\\
\subfigure[Crossing a negative cusp (cf. {\cite[Figure 9.8]{PettersLevineWambsganss2001}}).]{
\begin{tikzpicture}
\usetikzlibrary{positioning,arrows}
\tikzset{arrow/.style={-latex}}
\draw [thick] plot [smooth] coordinates {(-4,-2) (4,-2)};
\draw [thick] plot [smooth] coordinates {(-4,-2) (-4,2)};
\draw [thick] plot [smooth] coordinates {(4,-2) (4,2)};
\draw [thick] plot [smooth] coordinates {(-4,2) (4,2)};
\draw [thick] plot [smooth] coordinates {(0,-2) (0,2)};

\draw[color=black] (-2.4,-1.25) node[below] {$\CC$};
\draw[color=black] (-3.7,2) node[below] {\LARGE $+$};
\draw[color=black] (-.5,-2) node[above] {\LARGE $-$};
\draw[color=black] (-1.95,1.4) node[above] {$z_1$};
\draw [line width=1.5pt,arrow, dotted] (-1.95,0) to (-1.95,1.5);
\draw[color=black] (-1.95,-1.4) node[below] {$z_2$};
\draw [line width=1.5pt,arrow, dotted] (-1.95,0) to (-1.95,-1.5);
\draw [line width=1.5pt,arrow, dotted] (-2,0) to (-3.5,0);
\draw [line width=1.5pt, loosely dashed] (-2,0) to (-.3,0);
\draw[color=black] (-3.4,0) node[left] {$z$};
\draw[color=black] (-2.25,0) node[above] {$z_0$};
\draw[fill=black](-1.95,0)circle(2pt);
\draw [thick] (-2.4,-2) arc (-23.5:23.5:5);

\draw[fill=black](3.5,0)circle(2pt);
\draw[fill=black](2.5,0)circle(2pt);
\draw[color=black] (2.5,0) node[above] {$f(z_0)$};
\draw[color=black] (3.5,0) node[above] {$\eta$};
\draw[color=black] (3.5,2) node[below] {$N$};
\draw[color=black] (.6,-.7) node[above] {$N+2$};
\draw [line width=1.5pt, loosely dashed] (2.5,0) to (3.5,0);
\draw [line width=1.5pt,arrow, dotted] (2.5,0) to (.5,0);
\draw [thick] (2.5,0) arc (270:180:2);
\draw [thick] (2.5,0) arc (90:180:2);
\draw[color=black] (0.5,-1.5) node[right] {$f(\CC)$};
\end{tikzpicture}}
\caption{Crossing a caustic at a cusp point.}\label{fig:schema_cusp}
\end{figure}
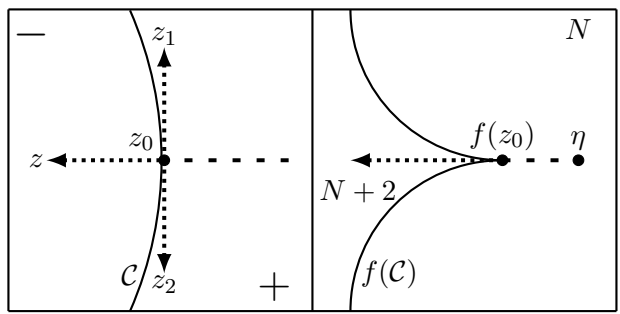
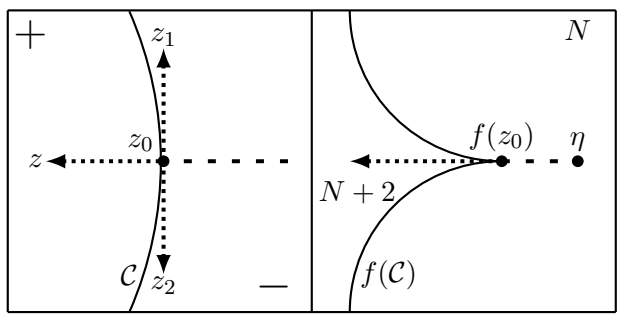 
\section{Examples and outlook}\label{sect:examples}

Let us give some examples that illustrate the results of the previous sections. 

First we consider the function 
$$
f(z) = \frac{z^{n-1}}{z^n - \rho^n} - \conj{z}
$$
for some $\rho > 0$. 
Functions of this form have been frequently studied in the context of gravitational 
lensing; see, e.g., the original work of Mao, Petters and Witt~\cite{MaoPettersWitt1997}, and the more 
recent articles~\cite{LuceSeteLiesen2014a,SeteLuceLiesen2015b}, which contain many further references. We 
choose $n = 3$ and $\rho = \frac{3}{5}$, and plot the zeros of $f_{\eta}$ for several constant shifts 
$\eta$ in \Cref{fig:exp_mpw}.

\begin{figure}[ht]
\begin{center}
\includegraphics[width=\textwidth]{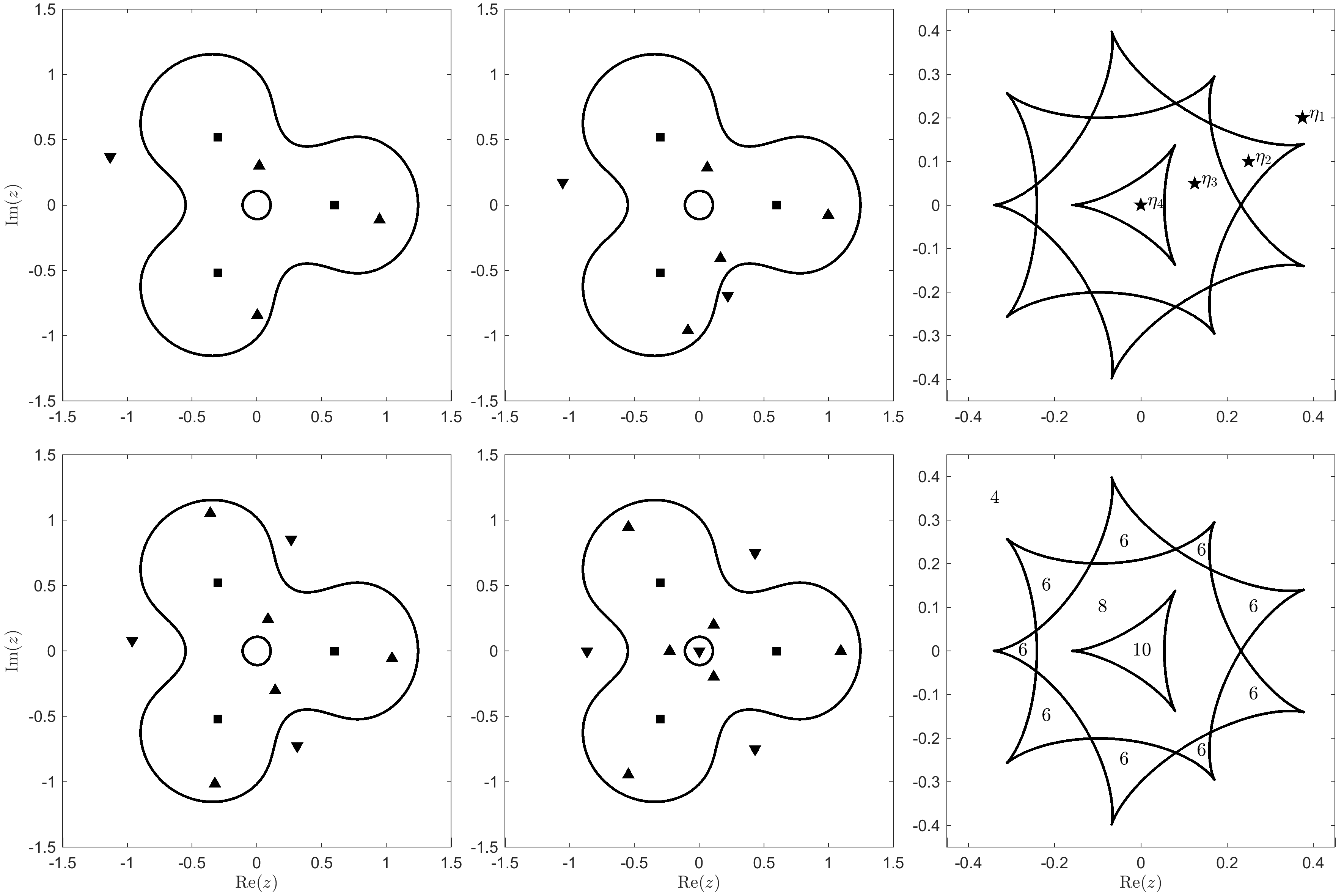}
\caption{Critical curves, zeros ($\blacktriangle$,$\blacktriangledown$) and poles (\footnotesize$\blacksquare$\normalsize) of $f_{\eta_1}$ (top left), $f_{\eta_2}$ (top mid), $f_{\eta_3}$ (bottom left) and $f_{\eta_4}$ (bottom mid); caustics (top 
right) and the number of zeros depending on the constant term $\eta_j$ (bottom right).}\label{fig:exp_mpw}
\end{center}
\end{figure}

We know from Theorem~\ref{thm:initial_value}, that for $|\eta| \gg 1$ the function $f_{\eta}$ has $3$ zeros close to its $3$ poles, 
and one zero in the set $A_\infty$. This can be observed for the shift $\eta_1$. The shift from $\eta_1$ to $\eta_2$ results 
in a caustic crossing with one additional pair of zeros (one sense-preserving and one sense-reversing) appearing at the outer critical 
curve, as predicted by Theorem~\ref{thm:main_fold} and the curvature of the caustic. The same happens when shifting from $\eta_2$ 
to $\eta_3$. Finally, the shift from $\eta_3$ to $\eta_4$ results in an additional pair of zeros at the inner critical curve.
Note that $N(f_{\eta_4}) = 10 = 5n-5$. Hence $f_{\eta_4}$ is an extremal rational harmonic function,
and it has $6=4n-6$ more zeros than $f_{\eta_1}$; cf. Remark~\ref{rmk:extremal}.

It was shown in~\cite[Theorem 3.1]{LuceSeteLiesen2014b} (see also~\cite[Theorem~3.5]{LiesenZur2018}), 
that an extremal rational harmonic function is always regular,
i.e., has no singular zeros. Our results in \Cref{sect:main_section_2} yield the following slight generalization.

\begin{lem}\label{lem:maximal_function}
Let $f$ be as in \eqref{eqn:def_f}, and suppose that there exists an $\eps>0$ with $N(f) \ge N(f_\eta)$ for 
all $\eta \in B_\eps(0)$. Then $f$ is regular.
\end{lem}

\begin{proof}
The function $f$ is singular if and only if $z=0$ is a caustic (fold or cusp) point of $f$. By the 
Theorems \ref{thm:main_fold} (fold case) and \ref{thm:main_cusp} (cusp case) there exist some $\eta \in \C$ 
such that $f_\eta$ has at least one additional zero, which contradicts the assumption $N(f) \ge N(f_\eta)$.
\eop
\end{proof}

Since an extremal rational harmonic function $f$ satisfies $N(f) \ge N(f_\eta)$ for all $\eta\in\C$,
Lemma~\ref{lem:maximal_function} immediately implies that $f$ must be regular. On the other hand, if $f$ is singular, then for every
$\eps > 0$ there must exist an $\eta \in B_\eps(0)$, such that $f_\eta$ is regular and $N(f_\eta) \le N(f)$.

\begin{figure}[ht]
\begin{center}
\includegraphics[width=\textwidth]{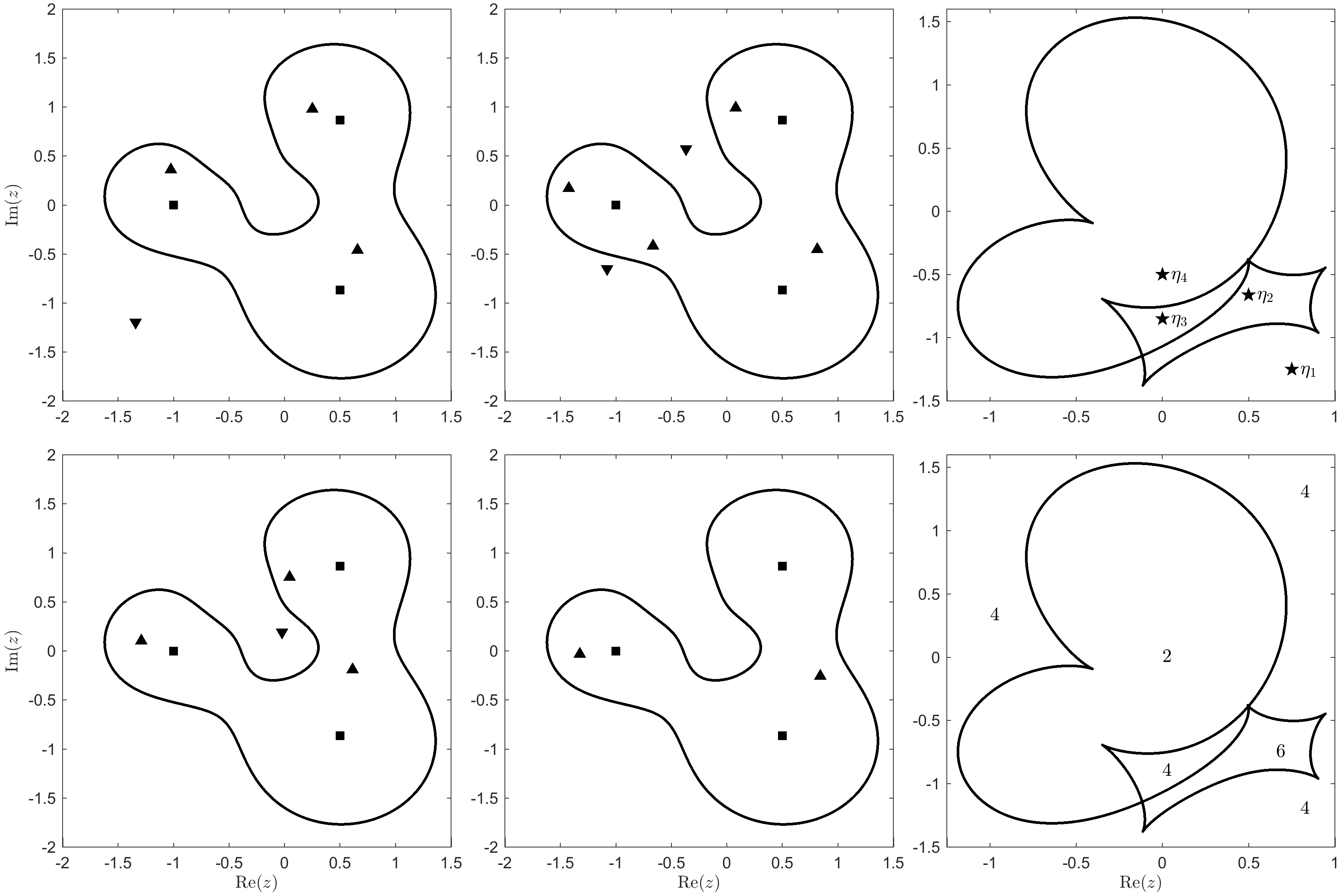}
\caption{Critical curves, zeros ($\blacktriangle$,$\blacktriangledown$) and poles (\footnotesize$\blacksquare$\normalsize) of $f_{\eta_1}$ (top left), $f_{\eta_2}$ (top mid), $f_{\eta_3}$ (bottom left) and $f_{\eta_4}$ (bottom mid); caustics (top 
right) and the number of zeros depending on the constant term $\eta_j$ (bottom right).}\label{fig:exp_new}
\end{center}
\end{figure}

As another example we consider
$$
f(z) = \frac{(1+i)z^2 - i}{z^3 + 1} - \conj{z},
$$
and plot the results in \Cref{fig:exp_new}. For $\eta_1$ we again have $3$ zeros close to the $3$ poles and one zero 
in $A_\infty$, as shown by Theorem~\ref{thm:initial_value}. The first caustic crossing from $\eta_1$ to $\eta_2$ results in 
one additional pair of zeros, but due to the curvature of the caustic, the shift from $\eta_2$ to $\eta_3$ reverses this 
effect. The last shift from $\eta_3$ to $\eta_4$ results again in two fewer zeros due to the curvature of the caustic,
giving $N(f_{\eta_4}) = 2$. Since $N_-(f_{\eta_4}) = 0$, we have a rational harmonic function with the \emph{minimal 
number of zeros}. (For $r=p/q$ with $\deg(p)\leq\deg(q)$ this number is $\deg(q)-1$, which can be easily proved using 
the argument principle.)

\medskip

Finally, we would like to mention that most of our theory in this paper can be extended from rational to general 
analytic functions, i.e., to functions of the form $f(z)=h(z)-\conj{z}$ with~$h$ being (locally) analytic. This is because
the derivation of our main results is based on the local Taylor series, and in the more general case we
we could start from $$h(z)=h'(z_0)(z-z_0)+\frac12 h''(z_0)(z-z_0)^2+ R(z,z_0).$$ A similar approach 
has recently been used in~\cite{LuceSete2017}. 

Another interesting extension would be to consider
rational harmonic functions of the form $r_1(z)-\conj{r_2(z)}$ with both $r_1$ and $r_2$ rational. We are
not aware of any general results on the zeros of such functions.

\paragraph*{Acknowledgements}
We thank Seung-Yeop Lee for sending us a pdf-file of~\cite{Wilmshurst1994}.

\footnotesize
\bibliography{literature}
\bibliographystyle{siam}

\end{document}